\documentclass{amsart}
\usepackage{amsfonts,amssymb,amscd,amsmath,enumerate,verbatim,calc}

\newcommand{\wrt}{with respect to}
\newcommand{\ZZ}{\mathbb{Z} }

\newcommand{\D}{\mathcal{D} }
\newcommand{\T}{\mathcal{T} }
\newcommand{\F}{\mathcal{F} }
\newcommand{\E}{\mathcal{E} }
\newcommand{\G}{\mathcal{G} }
\newcommand{\bX}{\mathbf{X} }
\newcommand{\bY}{\mathbf{Y} }

\newcommand{\N}{\mathcal{N} }
\newcommand{\n}{\mathfrak{n} }
\newcommand{\m}{\mathfrak{m} }
\newcommand{\M}{\mathcal{M} }
\newcommand{\q}{\mathfrak{q} }

\newcommand{\R}{\mathcal{R} }
\newcommand{\Sc}{\mathcal{S} }

\newcommand{\rt}{\rightarrow}

\newcommand{\ov}{\overline}

\newcommand{\FF}{\mathcal{T}}

\newcommand{\bP}{\mathbb{\partial}}

\newcommand{\Supp}{\operatorname{Supp}}
\newcommand{\injdim}{\operatorname{injdim}}
\newcommand{\gr}{\operatorname{gr}}

\newcommand{\content}{\operatorname{content}}
\newcommand{\height}{\operatorname{height}}

\newcommand{\Spec}{\operatorname{Spec}}
\newcommand{\coker}{\operatorname{coker}}

\newcommand{\Ass}{\operatorname{Ass}}
\newcommand{\Hom}{\operatorname{Hom}}
\newcommand{\Ext}{\operatorname{Ext}}

\theoremstyle{plain}

\newtheorem{theorem}{Theorem}[section]
\newtheorem{corollary}[theorem]{Corollary}
\newtheorem{lemma}[theorem]{Lemma}
\newtheorem{proposition}[theorem]{Proposition}
\newtheorem{question}[theorem]{Question}

\theoremstyle{definition}
\newtheorem{definition}[theorem]{Definition}
\newtheorem{s}[theorem]{}
\newtheorem{remark}[theorem]{Remark}
\newtheorem{property}[theorem]{Property}
\newtheorem{example}[theorem]{Example}

\theoremstyle{remark}

\begin{document}

\title{Graded components of local cohomology modules }
\author{Tony.~J.~Puthenpurakal}
\date{\today}
\address{Department of Mathematics, IIT Bombay, Powai, Mumbai 400 076}

\email{tputhen@math.iitb.ac.in}
 \begin{abstract}
Let $A$ be a regular ring containing a field of characteristic zero and let $R = A[X_1,\ldots, X_m]$. Consider $R$ as standard
graded with $\deg A = 0$ and $\deg X_i = 1$ for all $i$. In this paper we present a comprehensive study of graded components of
local cohomology modules
$H^i_I(R)$ where $I$ is an \emph{arbitrary} homogeneous ideal in $R$. Our study seems to be the first in this regard. 
\end{abstract}
\dedicatory{Dedicated to Prof. Gennady Lyubeznik}
 \maketitle
 \tableofcontents
\newpage
\section{Introduction}
Let $S = \bigoplus_{n \geq 0}S_n$ be a standard graded Noetherian ring and let $S_+$ be it's irrelevant ideal. The theory of local cohomology with respect to $S_+$ is particularly satisfactory. It is well-known (cf. \cite[15.1.5]{BS}) that if $M$ is a finitely generated 
graded $S$-module then  for all $i \geq 0$
\begin{enumerate}
\item
$H^i_{S_+}(M)_n$ is a finitely generated $S_0$-module for all $n \in \ZZ$.
\item
$H^i_{S_+}(M)_n = 0$ for all $n \gg 0$.
\end{enumerate}
It is natural to expect whether local cohomology with respect to other homogeneous 
ideals exhibit similar results (or predictable results). It has been well-known for many years that the answer to the latter question is clearly in the negative
(even in the case when $S = B[X_1,\ldots, X_n]$ is a polynomial ring). Even in the case of $S_+$ the
local cohomology module $H^i_{S_+}(M)$ need not be \emph{tame}, i.e., $H^i_{S_+}(M)_n \neq 0$ for infinitely many $n < 0$
does \emph{not} 
imply that $H^i_{S_+}(M)_n \neq 0$ for all $n \ll 0$, see \cite[2.2]{CH}. 

The purpose of this paper is to show that if $A$ is a regular ring containing a field of characteristic zero  and 
if $R = A[X_1,\ldots, X_n]$ is standard graded ( with $\deg A = 0$) then the theory of local cohomology of
$R$ with respect to \emph{arbitrary} homogeneous ideals of $R$ exhibit striking good behavior.
We should note that local cohomology modules over regular rings does indeed show good behavior.
For instance see the remarkable papers \cite{HuSh}, \cite{Lyu-1} and \cite{Lyu-2}. However
there has been no previous study of graded components of graded local cohomology modules of polynomial rings over regular rings.

\s \label{std} \emph{Standard Assumption:} From henceforth $A$ will denote a regular ring containing a field of characteristic zero. Let $R = A[X_1,\ldots, X_m]$ be standard graded with $\deg A = 0$ and $\deg X_i = 1$ for all $i$. We also assume $m \geq 1$. Let $I$ be a homogeneous ideal in $R$. Set $M = H^i_I(R)$. It is well-known that $M$ is a graded $R$-module. Set $M = \bigoplus_{n \in \ZZ}M_n$.

We first give a summary of the results proved in this paper.

 \textbf{I:} \textit{(Vanishing:)}  The first result we prove is that vanishing of almost all graded components of $M$ implies vanishing of $M$. More precisely we show
 \begin{theorem}\label{vanish}
(with hypotheses as in \ref{std}).  If $M _n = 0$ for all $|n|  \gg 0$ then  
$M = 0$.
 \end{theorem}
 
\textbf{II} (\textit{Tameness:}) In view of Theorem \ref{vanish}, it follows that if
$M = H^i_I(R) = \bigoplus_{n \in \ZZ}M_n $ is \emph{non-zero} then either $M_n \neq 0$ for infinitely  many $n \ll 0$, OR, $M_n \neq 0$ for infinitely  many $n \gg 0$. We show that $M$ is \emph{tame}. More precisely 
\begin{theorem}\label{tame}
(with hypotheses as in \ref{std}).  Then we have
\begin{enumerate}[\rm (a)]
\item
The following assertions are equivalent:
\begin{enumerate}[\rm(i)]
\item
$M_n \neq 0$ for infinitely many $n \ll 0$.
\item
$M_n \neq 0$ for all $n \leq -m$.
\end{enumerate}
\item
The following assertions are equivalent:
\begin{enumerate}[\rm(i)]
\item
$M_n \neq 0$ for infinitely many $n \gg 0$.
\item
$M_n \neq 0$ for all $n \geq 0$.
\end{enumerate}
\end{enumerate}
\end{theorem}
We complement Theorem \ref{tame} by showing the following
\begin{example}\label{ex-tame} There exists a regular ring $A$ and homogeneous ideals $I, J, K, L$ in $R = A[X_1,\ldots, X_m]$ such that
\begin{enumerate}[\rm (i)]
\item
$H^i_I(R)_n \neq 0$ for all $n \leq -m$ and $H^i_I(R)_n = 0$ for all $n > - m$.
\item
$H^i_J(R)_n \neq 0$ for all $n \geq 0$ and $H^i_I(R)_n = 0$ for all $n < 0$.
\item
$H^i_K(R)_n \neq 0$ for all $n \in \ZZ$.
\item
$H^i_L(R)_n \neq 0$ for all $n \leq -m$, $H^i_L(R)_n \neq 0$ for all $n \geq 0$ and $H^i_L(R) = 0$ for all $n$ with $-m < n < 0$.
\end{enumerate}
\end{example}

An easy way to construct examples of type (ii) above is as follows: Choose an ideal $Q$ in $A$ with $H^i_Q(A) \neq 0$. Then $H^i_{QR}(R) = H^i_Q(A)\otimes_A R$
will satisfy (ii). 
The author was also able to construct example of a homogeneous ideal $J$ in $R$ and ideal $Q$ in $A$ with $J \varsubsetneq  QR$ such that (ii) is satisfied and $\sqrt{J} \neq \sqrt{\mathbf{q}R}$ for any ideal $\mathbf{q}$ in $A$. Surprisingly the following general result holds:
\begin{theorem}\label{type2}
(with hypotheses as in \ref{std}). Further assume $A$ is a domain.
Suppose $J$ is a proper homogeneous ideal in $R$ such that $H^i_J(R)_n \neq 0$ for all $n \geq 0$ and $H^i_J(R)_n = 0$ for all $n < 0$. Then there exists a proper ideal $Q$ of $A$ such that $J \subseteq QR$.
\end{theorem}

\textbf{III} (\textit{Rigidity:})  Surprisingly  non-vanishing of a single  graded component of $M = H^i_I(R)$ is 
very strong. We prove the following rigidity result:
\begin{theorem}\label{rigid}
(with hypotheses as in \ref{std}).  The we have
\begin{enumerate}[\rm (a)]
\item
The following assertions are equivalent:
\begin{enumerate}[\rm(i)]
\item
$M_r \neq 0$ for  some $r \leq -m$.
\item
$M_n \neq 0$ for all $n \leq -m$.
\end{enumerate}
\item
The following assertions are equivalent:
\begin{enumerate}[\rm(i)]
\item
$M_s \neq 0$ for some $s \geq 0$.
\item
$M_n \neq 0$ for all $n \geq 0$.
\end{enumerate}
\item
(When $m \geq 2$.) 
The following assertions are equivalent:
\begin{enumerate}[\rm(i)]
\item
$M_t \neq 0$ for some $t$ with $-m < t < 0$.
\item
$M_n \neq 0$ for all $n \in \ZZ$.
\end{enumerate}
\end{enumerate}
\end{theorem}

\textbf{IV} \textit{(Infinite generation:)} Recall that each component of $H^m_{R_+}(R)$ is a finitely generated $R$-module, cf., \cite[15.1.5]{BS}. We give a sufficient condition for infinite generation of a component of graded local cohomology module over $R$.
\begin{theorem}\label{inf-gen}(with hypotheses as in \ref{std}). Further assume $A$ is a domain. Assume $I \cap A \neq 0$. If $H^i_I(R)_c \neq 0$ then
$H^i_I(R)_c$ is NOT finitely generated as an $A$-module.
\end{theorem}

\textbf{V} (\textit{Bass numbers:}) The $j^{th}$ Bass number of an $A$-module $E$ with respect to a prime ideal $P$ is defined as $\mu_j(P,E) = \dim_{k(P)} \Ext^j_{A_P}(k(P), E_P)$ where $k(P)$ is the residue field of $A_P$. We note that if $E$ is finitely generated as an $A$-module  then $\mu_j(P,E)$ is a finite number (possibly zero) for all $j \geq 0$. In view of Theorem \ref{inf-gen} it is not clear whether $\mu_j(P, H^i_I(R)_n)$ is a finite number. Surprisingly  we have the following dichotomy:
\begin{theorem}
\label{bass-basic}(with hypotheses as in \ref{std}).  Let $P$ be a prime ideal in $A$. Fix $j \geq 0$. EXACTLY one of the following hold:
\begin{enumerate}[\rm(i)]
\item
$\mu_j(P, M_n)$ is infinite for all $n \in \ZZ$.
\item
$\mu_j(P, M_n)$ is finite for all $n \in \ZZ$. In this case EXACTLY one of the following holds:
\begin{enumerate}[\rm (a)]
\item
$\mu_j(P, M_n) = 0$ for all $n \in \ZZ$.
\item
$\mu_j(P, M_n) \neq 0$ for all $n \in \ZZ$.
\item
$\mu_j(P, M_n) \neq 0$ for all $n  \geq 0$ and $\mu_j(P, M_n) = 0$ for all 
$n < 0$.
\item
$\mu_j(P, M_n) \neq 0$ for all $n  \leq -m$ and $\mu_j(P, M_n) = 0$ for all 
$n > -m$.
\item
$\mu_j(P, M_n) \neq 0$ for all $n  \leq -m$, $\mu_j(P, M_n) \neq 0$ for all $n  \geq 0$  and $\mu_j(P, M_n)  = 0$ for all $n$ with $-m < n < 0$.
\end{enumerate}
\end{enumerate}
\end{theorem}
We also give easy examples where (i) and (ii) hold.  The only examples where the author was able to show (i) hold had $m \geq 2$.  Surprisingly  the following result holds.
\begin{theorem}
\label{bass-m-one}(with hypotheses as in \ref{std}).  Assume $m = 1$. Let $P$ be a prime ideal in $A$.  Then 
$\mu_j(P, M_n)$ is finite for all $n \in \ZZ$.
\end{theorem}
\begin{remark}\label{intr}
An intriguing consequence of Theorem \ref{bass-basic} is the following: Suppose
$M = H^i_I(M) \neq 0$ \emph{but} $M_c = 0$ for some $c$. Then for any prime ideal $P$ and $j \geq 0$ we have $\mu_j(P, M_c) = 0$ \emph{is finite}. So $\mu_j(P, M_n) < \infty$ for all $n \in \ZZ$.
\end{remark}
\textbf{VI} (\textit{ Growth of Bass numbers}). Fix $j \geq 0$. Let $P$ be a prime ideal in $A$ such that $\mu_j(P, H^i_I(R)_n)$ is finite for all $n \in \ZZ$. We may ask about the growth of the function $n \mapsto \mu_j(P, H^i_I(R)_n)$ as $n \rt -\infty$ and when $n \rt + \infty$. We prove
\begin{theorem}
\label{bass-growth}(with hypotheses as in \ref{std}).  Let $P$ be a prime ideal in $A$. Let $j \geq 0$. Suppose $\mu_j(P, M_n)$ is finite for all $n \in \ZZ$. Then there exists polynomials $f_M^{j,P}(Z), g_M^{j,P}(Z) \in \mathbb{Q}[Z]$ of degree $\leq m - 1$ such that
\[
f_M^{j,P}(n) = \mu_j(P, M_n) \ \text{for all} \ n \ll 0  \quad \text{AND} \quad  g_M^{j,P}(n) = \mu_j(P, M_n) \ \text{for all} \ n \gg 0.
\]
\end{theorem}
By \ref{intr} if $M_c = 0$ for some $c$ then $\mu_j(P, M_n)$ is finite for all $n \in \ZZ$, $j \geq 0$ and  prime $P$ of $A$. In this case we prove:
\begin{theorem}\label{bass-growth-max}
(with hypotheses as in \ref{std}).  Let $P$ be a prime ideal in $A$. Fix $j \geq 0$. Suppose $\mu_j(P,M_c) = 0$ for some $c$ (this holds if for instance $M_c =0$). Then
\begin{align*}
f_M^{j,P}(Z) = 0  \quad &\text{or} \quad   \deg f_M^{j,P}(Z) = m -1, \\
g_M^{j,P}(Z) = 0  \quad &\text{or} \quad   \deg g_M^{j,P}(Z) = m -1.
\end{align*}
\end{theorem}
\textbf{VII}
(\textit{Associate primes:}) If $E = \bigoplus_{n \in \ZZ} E_n$ is a graded $R$-module then  there are two questions regarding asymptotic primes:

\textit{Question 1:}\textit{(Finiteness:)} Is the set 
$\bigcup_{n \in \ZZ} \Ass_A E_n   $ finite?

\textit{Question 2:} \textit{(Stability:)} Does there exists integers $r , s$ such
that $\Ass_A E_n = \Ass_A E_r$ for all $n \leq r$ and $\Ass_A E_n = \Ass_A E_s$ for all $n \geq s$.

For graded local cohomology modules we show that both Questions above have affirmative answer for a large class of regular rings $A$. 
\begin{theorem}\label{ass}(with hypotheses as in \ref{std}). Further assume that either $A$ is local or a smooth affine algebra over a field $K$ of characteristic zero. Let $M = H^i_I(R) = \bigoplus_{n \in \ZZ}M_n$. Then
\begin{enumerate}[\rm (1)]
\item
$\bigcup_{n \in \ZZ} \Ass_A M_n   $ is a finite set.
\item
$\Ass_A M_n = \Ass_A M_m$ for all $n \leq -m$.
\item
$\Ass_A M_n = \Ass_A M_0$ for all $n \geq 0$.
\end{enumerate}
\end{theorem}

\textbf{VIII}
(\textit{Dimension of Supports  and injective dimension:}) Let $E$ be an $A$-module. Let
$\injdim_A E$ denotes the injective dimension of $E$. Also 
$\Supp_A E = \{ P \mid  E_P \neq 0 \ \text{and $P$ is a prime in $A$}\}$ is the support of an $A$-module $E$. 
By $\dim_A E $ we mean the dimension of $\Supp_A E$ as a subspace of $\Spec(A)$.
We prove the following:
\begin{theorem}\label{injdim-and-dim}(with hypotheses as in \ref{std}). Let $M = H^i_I(R) = \bigoplus_{n \in \ZZ}M_n$. Then we have
\begin{enumerate}[\rm (1)]
\item
$\injdim M_c \leq \dim M_c$ for all $c \in \ZZ$.
\item
$\injdim M_n = \injdim M_{-m}$ for all $n \leq -m$.
\item
$\dim M_n = \dim M_{-m}$ for all $n \leq -m$.
\item
$\injdim M_n = \injdim M_{0}$ for all $n \geq 0$.
\item
$\dim M_n = \dim M_{0}$ for all $n \geq 0$.
 \item
 If $m \geq 2$ and $-m < r,s < 0$ then
 \begin{enumerate}[\rm (a)]
 \item
 $\injdim M_r = \injdim M_{s}$  and $\dim M_r = \dim M_{s}$.
 \item
 $\injdim M_r \leq \min \{ \injdim M_{-m}, \injdim M_{0} \}$.
 \item
 $\dim M_r \leq \min \{ \dim M_{-m}, \dim M_{0} \}$.
 \end{enumerate}
\end{enumerate}
\end{theorem}

\emph{Techniques used to prove our results:}
We use three main techniques to prove our results:

\s \label{Weyl-mod}(a)  For the first  technique  $A$ is arbitrary regular ring containing a field of characteristic zero.  Let  $R = A[X_1,\ldots,X_m]$, graded with $\deg A = 0$ and $\deg X_i = 1$ for all $j$. Let $\Sc$ be the $m^{th}$ Weyl-algebra on $A$. We consider it a graded ring with $\deg A = 0$, $\deg X_i = 1$ and $\deg \partial_i = -1$. We note that $R$ is a graded subring of $\Sc$. If $E$ is a graded $\Sc$-module and $e$ is a homogeneous element of $E$ then set $|e| = \deg e$. 

Consider the Eulerian operator $\E = \sum_{i = 1}^{m} X_i \partial_i$. If $f \in R$ is homogeneous then it is easy to check that $\E f = |f| f$. We say a graded $\Sc$-module $W$ is \textit{Eulerain} if $\E w = |w| w$ for each homogeneous element $w$ of $W$. Notice $R$ is an Eulerian $\Sc$-module.  We say $W$ is \textit{generalized Eulerian} if for each
homogeneous $w$ of $W$ there exists $a$ depending on $w$ such that
$(\E - |w|)^aw = 0$.

The notion of Eulerian modules was introduced in the case $A$ is a field $K$ by Ma and Zhang \cite{MaZhang} (they also defined the notion of Eulerian  $D$-modules in characteristic $p > 0$, where $D$ is the ring of $K$-linear  differential operators on $R = K[X_1,\ldots, X_m]$). Unfortunately however the class of Eulerian $D$-modules is not closed under extensions (see 3.5(1) in \cite{MaZhang}). To rectify this,   the author introduced the notion generalized Eulerian $D$-modules (in characteristic zero), see \cite{P2}.

One can define the notion of graded Lyubeznik functors on $\ ^*Mod(R)$ the category of all graded $R$-modules, see \ref{defn-grade-Lyu-functor}. 
Our techniques in \cite[ Theorem 1.7]{PS}  generalize to prove the following:
\begin{theorem}\label{third}[with hypotheses as in \ref{Weyl-mod}.]
 Let $\mathcal{G}$ be a graded Lyubeznik functor on $^*Mod(R)$. Then  $\mathcal{G}(R)$ is a generalized Eulerian $\Sc$-module.  
\end{theorem}

\s \label{D-mod}(b) For the second technique we look at $A= K[[Y_1,\ldots, Y_d]]$ where $K$ is a field of characteristic zero. Let  $R = A[X_1,\ldots,X_m]$, graded with $\deg A = 0$ and $\deg X_i = 1$ for all $j$. Let $D_k(A) = A<\delta_1,\ldots, \delta_d>$ (where $\delta_j = \partial /\partial Y_j$) be the ring of $k$-linear differential operators on  $A$. Furthermore set $\D = A_m(D_k(A))$ be the $m^{th}$-Weyl algebra over $D_k(A)$. We
can consider $\D$ a graded ring with $\deg D_k(A) = 0$, $\deg X_i = 1$ and $\deg \partial_i = -1$.  We note that $R$ is a graded subring of $\D$. 

It is well-known that the global dimension of $\D$ is $d + m$, see \cite[3.1.9]{Bjork}. Furthermore there is
a filtration $\mathcal{T}$ of $\D$ such that the associated graded $\gr_\T(\D)$ is the polynomial ring over $A$ with
$d + 2m$-variables, \cite[3.1.9]{Bjork}. Thus a  $\D$-module $M$ is holonomic if either it is zero  or there a
$\T$ compatible filtration $\F$ of $M$ such that 
$\gr_\F M$ is a finitely generated $\gr_\T \D$-module of dimension $d + m$. 

The main technical result we show is
\begin{theorem}\label{m-tech-cor}[with hypotheses as in \ref{D-mod}.]
Let $\mathcal{G}$ be a graded Lyubeznik functor on $^*Mod(R)$. Then  $\mathcal{G}(R)$ is a graded holonomic generalized Eulerian $\D$-module.
\end{theorem}

\s \label{deRham}(c) The final technique that we use is the technique of de Rham cohomology,
Koszul homology of generalized Eulerian modules (when $A = K$, a field).
We also prove
\begin{theorem}\label{kos-hom}
Let $K$ be a field of characteristic zero.
Let $A = K[[Y_1,\ldots, Y_d]]$ and $R = A[X_1,\ldots,X_n]$ and $\D$ is as in \ref{D-mod}. Let $E$ be a graded holonomic generalized Eulerian $\D$-module.
Then for all $i \geq 0$ the Koszul homology $H_i(Y_1,\ldots, Y_d; E)$ is a graded holonomic  generalized Eulerian $A_m(K)$-module (here $A_m(K)$ is the $m^{th}$-Weyl algebra over $K$).
\end{theorem}
 
 Our technique to prove all our results is as follows:
 Localize $A$ at an appropriate prime ideal and complete it. Then we use Theorem \ref{third}, Corollary \ref{m-tech-cor}
 and  Theorem \ref{kos-hom} to reduce to the case $A$ is a field. 
 
 \begin{remark}
  A natural question is what happens when $A$ contains a field of characteristic $p > 0$. We note that technique (a) and
  (b) have analogues in this case. However we do not know any analogue of our technique (c). So our proofs do not work in
  this case.
 \end{remark}

\textit{Acknowledgment:} In 2008, I started learning about applications of $D$-modules in local cohomology theory.
At that time Prof. G. Lyubeznik visited IIT-Bombay. I asked him whether de Rham cohomology 
of local cohomology modules will be interesting. He told me that it will be of interest. I  (and co-authors) developed 
techniques to study de Rham cohomology and Koszul cohomology  of local cohomology modules
in a series of papers \cite{P1}, \cite{P2}, \cite{PR1}, \cite{PR2} and \cite{PS}. These techniques have
proved to be fantastically useful in this paper. I thank
Prof. G. Lyubeznik for his advice and to him this paper is dedicated.

 \section{Preliminaries}
In this section, we discuss a few preliminary result that we need. 

 \s \textbf{Lyubeznik functors:} \\
 Let $B$ be a commutative Noetherian ring and let $X = \Spec(B)$. Let $Y$ be a locally closed subset of $X$.
 If $M$ is a $B$-module and  $Y$ be a locally closed subscheme of $\Spec(R)$, we denote by $H^i_Y(M)$ the
 $i^{th}$-local cohomology module of $M$ with support in $Y$.  Suppose 
 $Y = Y_1 \setminus Y_2$ where $Y_2 \subseteq Y_1$ are two closed subsets of $X$ then we have an exact sequence of functors
 \[
 \cdots \rt H^i_{Y_1}(-) \rt H^i_{Y_2}(-) \rt H^i_Y(-) \rt H^{i+1}_{Y_1}(-) \rt .
 \]
 A Lyubeznik functor $\FF$ is any functor of the form $\FF = \FF_1\circ \FF_2 \circ \cdots \circ \FF_m$ where every functor $\FF_j$  is either $H^i_Y(-)$ for some locally closed subset of $X$ or the kernel, image or
cokernel of some arrow in the previous long exact sequence for closed
subsets $Y_1,Y_2$ of $X$  such that $Y_2 \subseteq Y_1$.

\s \textit{Lyubeznik functor under flat maps:}\\
We need the following result from \cite[3.1]{Lyu-1}.
\begin{proposition}\label{flat-L}
 Let $\phi \colon B \rt C$ be a flat homomorphism of Noetherian rings. Let $\FF$ be a 
 Lyubeznik functor on $Mod(B)$. Then there exists a Lyubeznik functor $\widehat{\FF}$ on $Mod(C)$ and
 isomorphisms $\widehat{\FF}(M\otimes_B C) \cong \FF(M)\otimes_B C$ which is functorial in $M$.
\end{proposition}

\s  \textbf{Graded Lyubeznik functors:} \\
Let $B$ be a commutative Noetherian ring and let $R = B[X_1,\ldots, X_m]$ be standard graded.
We say $Y$ is \textit{homogeneous }closed subset of $\text{Spec}(R)$ if 
$Y= V(f_1, \ldots, f_s)$, where $f_i's$ are homogeneous polynomials in $R$.

We say $Y$ is a homogeneous locally closed subset of $\text{Spec}(R)$ if $Y=Y''-Y'$, where $Y', Y''$ 
are homogeneous closed subset of $\text{Spec}(R)$.  Let $\ ^*Mod(R)$ be the category of graded $R$-modules. 
We have an exact sequence of  functors on $\ ^*Mod(R)$,
\begin{equation}\label{eq1} H_{Y'}^i(-) \longrightarrow H_{Y''}^i(-) \longrightarrow H_{Y}^i(-) \longrightarrow 
H_{Y'}^{i+1}(-).
\end{equation}

\begin{definition}\label{defn-grade-Lyu-functor}
\textit{A graded Lyubeznik functor} $\mathcal{T}$ is a composite functor of the form 
$\mathcal{T}= \mathcal{T}_1\circ\mathcal{T}_2 \circ \ldots\circ\mathcal{T}_k$, where 
each $\mathcal{T}_j$ is either $H_{Y_j}^i(-)$, where $Y_j$ is a homogeneous locally closed subset of $\text{Spec}(R)$
or the kernel of any arrow appearing in (\ref{eq1}) with $Y'=Y_j'$ and $Y''= Y_j''$, where $Y_j' \subset Y_j''$ are two homogeneous  closed subsets of $\text{Spec}(R)$.
\end{definition}

\s \label{std-op} \textit{Graded Lyubeznik functors \wrt \ some standard operations on $B$.}

(a) Let $S$ be multiplicatively closed set in $B$. Set 
$$R_W = R\otimes_B B_W = B_W[X_1,\ldots, X_m].$$
If $Y$ is a 
homogeneous closed subset of $\text{Spec}(R)$, say
$Y= V(f_1, \ldots, f_s)$, set $Y_W = V(f_1/1, \ldots, f_s/1)$  where $f_i/1$ is the image of $f_i$ in $ R_W$.
We note that
$Y_W$ is a homogeneous closed subset of $\Spec(R_W)$. Furthermore it is clear that we have a homogeneous isomorphism
$H^i_Y(-)\otimes_B B_W \cong H^i_{Y_W}(-)$. If $Y$ is a homogeneous locally closed subset of $\Spec(R)$, say  $Y = Y'' - Y'$. set 
$Y_W = Y''_W - .Y'_W$, a homogeneous locally closed subset of $\Spec(R_W)$. Furthermore localizing (\ref{eq1})
at $W$ yields a
homogeneous isomorphism $H^i_{Y_W}(-) \cong H^i_Y(-)\otimes_B B_W$. More generally if $\FF$ is a graded Lyubeznik functor
on $R$ then $\FF\otimes_B B_W$ is a graded Lyubeznik functor on $ \ ^*Mod(R_W)$.

(b) Assume $B$ is local with maximal ideal $\m$. Let $\widehat{B}$ be the completion of $B$ \wrt \ $\m$. Set 
$$\widehat{R} = R \otimes_B \widehat{B} = \widehat{B}[X_1,\ldots, X_m].$$
If $Y$ is a 
homogeneous closed subset of $\text{Spec}(R)$, say
$Y= V(f_1, \ldots, f_s)$, set $\widehat{Y} = V(\widehat{f_1}, \ldots, \widehat{f_s})$  where $\widehat{f_i}$ is the image of
$f_i$ in $\widehat{R}$. We note that
$\widehat{Y}$ is a homogeneous closed subset of $\Spec(R_W)$. Furthermore it is clear that we have a homogeneous isomorphism
$H^i_Y(-)\otimes_B \widehat{B} \cong H^i_{\widehat{Y}}(-)$. If $Y$ is a homogeneous locally closed subset of
$\Spec(R)$, say  $Y = Y'' - Y'$. set 
$\widehat{Y} = \widehat{Y''} - \widehat{Y'}$, a homogeneous locally closed subset of $\Spec(\widehat{R})$. Furthermore applying the functor $(-)\otimes_B \widehat{B}$ to (\ref{eq1})
 yields a
homogeneous isomorphism $H^i_{\widehat{Y}}(-) \cong H^i_Y(-)\otimes_B \widehat{B}$. More generally if $\FF$ is a graded Lyubeznik functor
on $R$ then $\FF\otimes_B \widehat{B}$ is a graded Lyubeznik functor on $ \ ^*Mod(\widehat{R})$.

\s \textbf{Weyl Algebra's:} \\
Let $\Gamma$ be a ring (not necessarily commutative). The first Weyl algebra over $\Gamma$ is denoted by 
$A_1(\Gamma)$ and it is 
the ring $\Gamma<x,y>/(xy - yx -1)$.  Alternatively we can consider the polynomial ring $\Gamma[X]$ and let 
$\delta$ be the derivation on $\Gamma[X]$ defined by formal differentiation with respect to $X$ (treating elements of
$\Gamma$ as constants):
\[
 \delta\left(\sum \gamma_i X^i \right)  = \sum i \gamma_i X^{i-1}.
\]
It can be easily shown that the differential polynomial ring $\Gamma[X][Y; \delta]$ is isomorphic to the first Weyl algebra
over $\Gamma$, see \cite[p.\ 12]{Lam}. We note that  $YX = XY + \delta(X) = XY + 1$. Thus $Y$ can be thought as $\partial/\partial X$. The identification
of $A_1(\Gamma)$ as $\Gamma[X][Y; \delta]$ also yields that a $\Gamma$ basis of $A_1(\Gamma)$ is given by
\[
 \{ X^i Y^j \colon i,j \geq 0 \} \ \quad \text{as well as by} \quad \{ Y^jX^i \colon  i,j \geq 0 \}. 
\]

\noindent It follows that if $Z(\Gamma)$, the center of $\Gamma$,
contains a field $k$ then $A_1(\Gamma) = \Gamma \otimes_k A_1(k)$.

The higher Weyl algebra's are defined inductively as $A_m(\Gamma) = A_1(A_{m-1}(\Gamma)$.
\textit{For us the ring $\Gamma$  will always contain a field $K$ of characteristic zero in its center.} So
$A_m(\Gamma) = \Gamma \otimes_K A_m(K)$. The description of $A_m(K)$ can be found in \cite[Chapter 1]{Bjork}.

\s \textbf{ Koszul homology:} 
Let $\Gamma$ be a not-necessarily commutative, $\ZZ$-graded ring. We assume that $\Gamma$  has
a commutative field $K$ in its center with $\deg K = 0$
 Let $u_1,\ldots, u_c$ be homogeneous commuting 
elements in $\Gamma$. Consider the (commutative) subring 
 $S = K[u_1,\ldots, u_l]$ of $\Gamma$.  
 
Let $M$ be a graded $\Gamma$-module. 
Let  $H_i(u_1,\ldots, u_c; M)$ be the $i^{th}$ Koszul homology module of $M$ 
with respect to $u_1,\ldots, u_c $. It is clearly a graded $S$-module. 
(with its natural grading). The following result is well-known.
\begin{lemma}\label{exact}
Let $ \mathbf{u} = u_r,  u_{r+1}, \ldots, u_c$ and let $\mathbf{u^\prime}  = u_{r+1}, \ldots, u_c$. 
 For each $i\geq 0$ there exists an exact sequence of graded $S$-modules.

$$ 0 \rightarrow H_0(u_r; H_i(\mathbf{u^\prime}; M))\rightarrow H_i( \mathbf{u}; M) 
\rightarrow H_1(u_r; H_{i-1}(\mathbf{u^\prime} ; M)) \rightarrow 0.$$
\end{lemma}

\s \textit{Examples}
\begin{enumerate}
 \item $\Gamma = A_m(K)$, the $m^{th}$-Weyl algebra over $K$. The operators $\partial_1,\cdots, \partial_m$ commute with each other.
 In this case $H_i(\bP; M)$ is usually called the $i^{th}$ de Rham homology  of $M$.
 \item $\Gamma = A_m(K)$, the $m^{th}$-Weyl algebra over $K$. The elements $X_1,\ldots, X_m$ commute with each other.
 We can consider the usual Koszul homology modules $H_i(\bX, M)$.
\end{enumerate}

\s \label{inj-koszul} We will recall the following computation of Koszul homology
 which we need. Let $A = K[[Y_1,\ldots, Y_d]]$. Let $E$ be the injective hull of $K$ as an $A$-module.
 Then 
 \[
  H_\nu(\bY, E) = \begin{cases} 
                   K & \textit{if} \ \nu = d \\
                   0 & \textit{otherwise}.
                  \end{cases}
                  \]
Although we believe that this result is already known, we give sketch of a  proof for the benefit of the reader.
Set $A_i = K[[Y_1,\ldots, Y_i]]$ and $E_i$ to be the injective hull of $K$ as an $A_i$-module.
We have an exact sequence 
\[
 0 \rt A_i \xrightarrow{Y_i} A_i \rt A_{i-1} \rt 0
\]
Taking Matlis dual's \wrt \ $E_i$ yields
\[
 0 \rt E_{i-1} \rt E_i \xrightarrow{Y_i} E_i \rt 0.
\]
Now an easy induction of Koszul homology ( compute $H_j(Y_i, Y_{i+1},\cdots, Y_d; E_d)$) yields the result.

\s \label{min-loc} We will use the following well-known result often. Let $B$ be a Noetherian ring and let $M$ be an 
$A$-module not necessarily finitely generated. Let $P$ be a minimal prime of $M$. Then the $B_P$-module $M_P$ has a natural structure
of an $\widehat{B_P}$-module (here $\widehat{B_P}$ is the completion of $B_P$ \wrt \ it's maximal ideal $PB_P$. In fact
$M_P \cong  M_P\otimes_{B_P}\widehat{B_P}$.

\section{Generalized Eulerian modules}
The hypotheses in this section is a bit involved. So we will carefully state it.

\s \label{setup-gen-eul} \emph{Setup:}  In this section 
\begin{enumerate}
 \item $B$ is a commutative Noetherian ring containing a field $K$ of characteristic zero.
 \item
 We also assume that there  is a not necessarily commutative ring $\Lambda$ containing $B$ such
that $K \subseteq Z(\Lambda)$, the center of $\Lambda$. Furthermore we assume that
\begin{enumerate}
 \item $\Lambda$ is free both as a left $B$-module and as a right $B$-module.
 \item $N = B$ is a left $\Lambda$-module such that if we restrict the $\Lambda$ action on $N$ to $B$ we get the usual
 action of $B$ on $N$.
\end{enumerate}
\item
Set $R = B[X_1,\ldots, X_m]$. We consider $R$ graded with $\deg B = 0$ and $\deg X_i = 1$ for all $i$.
\item
Set $\Gamma = \Lambda[X_1,\ldots, X_m]$. We consider $\Gamma$ graded with $\deg \Lambda = 0$ and $\deg X_i = 1$ for
all $i$. Notice
\begin{enumerate}
 \item $R$ is a graded subring of $\Gamma$.
 \item $\Gamma$ is free both as a left and right $R$-module.
 \item $N^\prime  = R$ is a graded left $\Gamma$-module such that if we restrict the $\Gamma$ action on $N^\prime$ to $R$ we get the usual
 action of $R$ on $N^\prime$.
\end{enumerate}
\item
We also assume that for each homogeneous ideal $I$ of $R$ and a graded $\Gamma$-module $E$ the set
$H^0_I(E)$ is a graded $\Gamma$-submodule of $E$. Here
\[
 H^0_I(E) = \{ e \in E \mid I^s e = 0 \ \text{for some} \ s \geq 1 \}.
\]
\item
Let $\D_m$ be the $m^{th}$-Weyl algebra over $\Lambda$. Note $\D_m = \Lambda \otimes_K A_m(K)$ where $A_m(K)$ is the $m^{th}$-Weyl algebra
over $K$. We can consider $\D_m$ graded by giving
$\deg \Lambda = 0$, $\deg X_i = 1$ and $\deg \partial_i = -1$. Notice
\begin{enumerate}
 \item $\Gamma$ is a graded subring of $\D_m$.
 \item $\D_m$ is free both as a left and right module over $\Gamma$. Thus $\D_m$ is free both as a left and right  $R$-module.
 \item The operators $\partial_i$ act as derivations on $R$. It follows that $R$ is a graded $\D_m$-module.
\end{enumerate}
\end{enumerate}

\s \label{ex-gen-eul} \textit{Examples where our hypotheses \ref{setup-gen-eul} hold:} 

\begin{enumerate}
 \item $K = B = \Lambda$.
 \item $B \neq K$ but $\Lambda = B$.
 \item $B = K[[Y_1,\ldots, Y_d]]$ and  $\Lambda = D_K(B)$ is the ring of $K$-linear differential operators on $B$, i.e.,
 $\Lambda = B<\delta_1, \cdots, \delta_d>$ where $\delta_j = \partial/\partial Y_j$. We have to verify hypotheses (5) of \ref{setup-gen-eul}.
 We note that in the action of $\Gamma$ on $R$ the elements $\delta_j$ act as derivations on $R$.
 Let $I$ be a homogeneous ideal of $R$ and let $E$ be a graded $\Gamma$-module. We note that
 $H^0_I(E)$ is a graded $R$-submodule
 of $E$. Let $e \in H^0_I(E)$. Say $I^s e = 0$.
 \begin{enumerate}
  \item We first show that $\delta_j e \in H^0_I(E)$. We claim $I^{s+1}\delta_j e = 0$. Let $u \in I^{s+1}$. Notice
  $u \delta_j = \delta_j u + \delta_j(u)$. We note that $\delta_j(u) \in I^s$. Thus $u \delta_j e = 0$. So
  our claim is true.
  \item
  We now note that $\Gamma$ is a free left $R$-module with generators 
  $$\{ \delta_1^{i_1} \delta_2^{i_2} \cdots \delta_d^{i_d} \mid i_1,i_2, \cdots, i_d \geq 0 \}.$$
  It follows that $H^0_I(E)$ is a $\Gamma$-submodule of $E$.
 \end{enumerate}

\end{enumerate}

\s \textit{Generalized Eulerian $\D_m$-modules:} \\
The Euler operator on $\D_m$, denoted by $\mathcal{E}_m$, is defined as 
$$ \E_m := \sum_{i=1}^m X_i\partial_i. $$
Note that $\deg \mathcal{E}_m  = 0$. Let $E$ be a graded $\D_m$-module. If $e \in E$ is homogeneous element, set $|e|= \deg e$.
\begin{definition}
 Let $E$ be a graded $\D_m$-module  Then $E$ is said to be  \textit{Eulerian} if    for each homogeneous element $e$
 of $E$.
$$ \mathcal{E}_m e= |e|\cdot e.$$ 
\end{definition}
We note that $R$ is an Eulerian $\D_m$-module.
\begin{definition}
A graded $\D_m$-module $M$ is said to be \textit{generalized Eulerian} if for each homogeneous element $e$ of $E$
there exists a positive integer $a$ (depending on $e$) such that
$$ (\mathcal{E}_m- {|e|})^a \cdot e = 0.$$
\end{definition}

The main result of this section is
\begin{theorem}\label{gen-eul-Lyu}
 Let $\FF$ be a graded Lyubeznik functor on $ \ ^* Mod(R)$. Then $\FF(R)$ is a generalized Eulerian $\D_m$-module.
\end{theorem}
The first step in proving Theorem \ref{gen-eul-Lyu} is that $\FF(R)$ is a graded $\D_m$-module. Notice that $\FF(R)$ is a graded $R$-module.
So we have to show both that $\FF(R)$ is a $\D_m$-module and that this action is compatiable with the grading on $\FF(R)$.
We isolate this fact as a seperate
\begin{lemma}\label{BHU}
 [with hypotheses as in Theorem \ref{gen-eul-Lyu}] $\FF(R)$ is a graded  $\D_m$-module.
\end{lemma}
\begin{remark}
 If $\Sc = \bigoplus_{n \in \ZZ} \Sc_n$ is a graded  but not-necessarily commutative ring then the category $\ ^* Mod(\Sc)$
 of graded left $\Sc$-modules has enough injectives (the proof given in Theorem 3.6.3 of \cite{BH} in the case
 $\Sc$ is commutative extends to the non-commutative case).
\end{remark}
The following result is an essential ingredient in proving Lemma \ref{BHU}.
\begin{proposition}\label{inj}
Let $V$ be a $ \ ^* $-injective left $\D_m$-module. Then $V$ is a $ \ ^* $-injective $R$-module.
 \end{proposition}
\begin{proof}
 We note that
 \begin{align*}
  \ ^* \Hom_R(- , V) &= \ ^* \Hom_R \left(- , \ ^* \Hom_{\D_m}( \D_m , V) \right), \\
  &= \ ^* \Hom_{\D_m}( \D_m \otimes_R -, V).
 \end{align*}
As $\D_m$ is free as a  graded right $R$-module we have that $\D_m\otimes_R -$ − is an exact functor
from  $ \ ^* Mod(R)$  to $\ ^* Mod(\D_m)$. Also by hypothesis 
$V$ is a $\ ^*$-injective $\D_m$-module. It
follows that  $ \ ^* \Hom_R(- , V) $ is an exact functor. So $V$  is $\ ^*$-injective as a $R$-module.
\end{proof}

\s We now give
\begin{proof}[Proof of Lemma \ref{BHU}] 

\emph{Step-1:}  Let $I$ be a homogeneous ideal in $R$. Let $M$ be a graded left $\D_m$-module and let $i \geq 0$ be fixed.
Then $H^i_I(M)$, the $i^{th}$-local
cohomology module of $M$ with respect to $I$, has a 
canonical structure
of a graded left $\D_m$-module.

\textit{Proof of Step 1:} 

We first show that $H^0_I(M)$ is a graded $\D_m$-submodule of $M$.
By our assumption \ref{setup-gen-eul}(5), $H^0_I(M)$ is a graded $\Gamma = \Lambda[X_1,\ldots, X_m]$-submodule of $M$.
Let $e \in H^0_I(M)$. Say $I^s e = 0$.
 \begin{enumerate}
  \item We first show that $\partial_j e \in H^0_I(E)$. We claim $I^{s+1}\partial_j e = 0$. Let $u \in I^{s+1}$. Notice
  $u \partial_j = \partial_j u - \partial_j(u)$. We note that $\partial_j(u) \in I^s$. Thus $u \partial_j e = 0$. So
  our claim is true.
  \item
  We now note that $\D_m$ is a free left $\Gamma$-module with generators 
  $$\{ \partial_1^{i_1} \partial_2^{i_2} \cdots \partial_m^{i_m} \mid i_1,i_2, \cdots, i_m \geq 0 \}.$$
  It follows that $H^0_I(M)$ is a $\D_m$-submodule of $M$.
 \end{enumerate}

Let $\mathbb{E}$ be a $ \ ^*$-injective resolution of $M$ as $\D_m$-module. Then by Proposition \ref{inj}
 we get that $\mathbb{E}$ be a $ \ ^*$-injective resolution of $M$ as $R$-module.
 So $H^i_I(M) = H^i(H^0_I(\mathbb{E}))$. But as shown earlier  we get that
 $H^0_I(\mathbb{E})$ is complex of graded $\D_m$-modules. So $H^i_I(M)$ has a structure of a $\D_m$-module. Standard arguments
 yield that this structure is independent of the resolution $\mathbb{E}$ of $M$ (as a $\D_m$-module).
 
 \emph{Step-2:} 
 Let $Y$ be a homogeneous locally closed subset of $\text{Spec}(R)$ i.e., $Y=Y''-Y'$, where $Y', Y''$ 
are homogeneous closed subset of $\text{Spec}(R)$.  Say $Y' = V(I)$ and $Y'' = V(J)$ where $I, J$ are homogeneous ideals in $R$
 with $J \subseteq I$. Let $\ ^*Mod(R)$ be the category of graded $R$-modules. 
We have an exact sequence of  functors on $\ ^*Mod(R)$,
\begin{equation}\label{eq1-2} H_{Y'}^i(-) \longrightarrow H_{Y''}^i(-) \longrightarrow H_{Y}^i(-) \longrightarrow 
H_{Y'}^{i+1}(-).
\end{equation}
We note that $\ ^* Mod(\D_m)$ the category of graded left $\D_m$-modules is a subcategory of  $\ ^*Mod(R)$. 

\textit{Claim-1:} (\ref{eq1-2}) is an exact sequence of functors on  $\ ^* Mod(\D_m)$.

To see this let $M$ be a graded left $\D_m$-module.
Let $\mathbb{E}$ be a $ \ ^*$-injective resolution of $M$ as $\D_m$-module. Then by Proposition \ref{inj}
 we get that $\mathbb{E}$ be a $ \ ^*$-injective resolution of $M$ as $R$-module. We have an exact sequence of complexes
 \begin{equation}\label{d-loc-closed}
  0 \rt H^0_I(\mathbb{E}) \xrightarrow{\phi} H^0_J(\mathbb{E}) \rt \mathbb{L} \rt 0,
 \end{equation}
where $\phi$ is the canonical inclusion and $\mathbb{L}$ is the quotient complex. We note that $H^i_Y(M) = H^i(\mathbb{L})$.
As argued in Step-1,
$H^0_I(\mathbb{E})$ and $H^0_J(\mathbb{E})$ are $\D_m$ sub-complexes of $\mathbb{E}$. So $\mathbb{L}$ is a complex of $\D_m$-modules.
Taking cohomology gives the desired result. This proves Claim-1.

 To prove the assertion of Lemma  \ref{BHU},     it suffices to show that if $\mathcal{T}$ is a graded Lyubeznik functor
and $M$ is a graded  $\D_m$-module, then so is $\mathcal{T}(M)$. 
Since $\mathcal{T}= \mathcal{T}_1\circ \mathcal{T}_2 \circ \ldots \circ \mathcal{T}_s$, 
by induction it is enough to show that $\mathcal{T}_i(M)$ is a graded $\D_m$-module.
By Claim-1;  $\mathcal{T}_i(M)$ is a graded $\D_m$-submodule of $H_Y^i(M)$, where $Y$ is locally closed homogeneous
closed subset of $\text{Spec}(R)$. 
\end{proof}

The following properties of generalized Eulerian modules were proved in \cite{P2} in the case $B = K = \Lambda$.
The  proofs in \cite{P2} generalize in the present setup \ref{setup-gen-eul}.

\begin{property}[Proposition 2.1 in \cite{P2}] \label{property1} Let
$0 \rightarrow M_1 \overset{{\alpha_1}}{\rightarrow} M_2 \overset{{\alpha_2}}{\rightarrow} M_3 {\rightarrow} 0 $ be a
short exact sequence of graded $\D_m$-modules. Then $M_2$ is generalized Eulerian
 if and only if $M_1$ and $M_3$ are generalized Eulerian.
\end{property}

If $M$ is graded $\D_m$-module, then for $l\in \mathbb{Z}$ the modules $M(l)$ denotes the shift of $M$ by $l$; that 
is, $M(l)_n=M_{n+l}$ for all $n\in \mathbb{Z}.$
\begin{property}[Proposition 2.2 in \cite{P2}]\label{property2}
Let $M$ be a non-zero generalized Eulerian $\D_m$-module. Then the shifted module $M(l)$ is not a generalized 
Eulerian $\D_m$-module for $l \neq 0$.
\end{property}
We note that the proof of Proposition 2.2  in \cite{P2} uses the fact that $K$ is a field of characteristic zero and that 
$K \subseteq Z(\D_m)$, the center of $\D_m$.

In \cite{PS} the following result was proved in the case $B = K = \Lambda$. The same proof generalizes in the present setup \ref{setup-gen-eul}.
\begin{property}\label{property4}
Let $M$ be a nonzero generalized Eulerian $\D_m$-module.
Let $S$ be a multiplicatively closed set of homogeneous elements in $R$.
Then $S^{-1}M$ is also a generalized Eulerian $\D_m$-module. 
 In particular, $M_f$ is generalized Eulerian for each homogeneous polynomial $f \in R$.
\end{property}
\begin{remark}\label{hom-gen-eul}
Computing local cohomology with respect to Cech-complex, cf. \cite[5.1.19]{BS}, we immediately get that if $M$ is generalized Eulerian $\D_m$-module 
and if $I$ is a homogeneous ideal of $R$ then $H^i_I(M)$ is generalized Eulerian for all $i \geq 0$.
\end{remark}

As an easy consequence of Lemma \ref{BHU} and \ref{property4} we get
\begin{proof}[ Proof of Theorem \ref{gen-eul-Lyu}:]   
 By proof of Lemma \ref{BHU} we get that if $M$ is a graded $\D_m$-module and $\FF$ is a graded Lyubeznik functor then 
 $\FF(M)$ is a graded $\D_m$-module. To prove our result it suffices to prove that if $M$ is generalized Eulerian
 $\D_m$-module then so is $\FF(M)$. Since $\mathcal{T}= \mathcal{T}_1\circ \mathcal{T}_2 \circ \ldots \circ \mathcal{T}_s$, 
by induction it is enough to show that $\mathcal{T}_i(M)$ is a generalized Eulerian $\D_m$-module.

But $\mathcal{T}_i(M)$ is a graded $\D_m$-submodule of $H_Y^i(M)$, where $Y$ is locally closed homogeneous
closed subset of $\text{Spec}(R)$. 
By Property \ref{property1}, it  suffices to show that $H_Y^i(M)$ is generalized Eulerian. 
Let $Y= Y''-Y'$, where $Y', Y''$ are homogeneous closed sets of $\text{Spec}(R)$. 
Then by Claim-1 of Lemma \ref{BHU} we have an exact sequence of graded $\D_m$-modules
\begin{equation*} 
H_{Y'}^i(M) \longrightarrow H_{Y}^i(M) \longrightarrow H_{Y''}^{i+1}(M). 
\end{equation*}
By Remark \ref{hom-gen-eul}, $H_{Y'}^i(M)$ and $H_{Y''}^{i+1}(M)$ are generalized Eulerian $\D_m$-modules.
Thus by Property \ref{property1}, $H_Y^i(M)$ is generalized Eulerian $\D_m$-module.
 \end{proof}

 \s \label{koszul-gen-eul}In this subsection we assume that the hypotheses in \ref{ex-gen-eul}(3) holds.
 Set $B^\prime = K[[Y_1,\ldots, Y_{d-1}]]$, $\Gamma^\prime = D_K(B^\prime)$ the ring of $K$-linear differential operators
 on $B^\prime$. Set $\D_m^\prime = A_m(\Gamma^\prime)$. We note that $B^\prime, \Gamma^\prime$ and $\D_m^\prime$ 
 are subrings of $A, \Gamma$ and $\D_m$-respectively. We note that the Eulerian operator $\E_m$ of $\D_m^\prime$
 extends to the Eulerian operator on $\D_m$. The main result of this subsection is
 \begin{proposition}\label{kos-induct}[ with hypotheses as in \ref{koszul-gen-eul}]
  Let $E$ be a generalized Eulerian $\D_m$-module. Then the Koszul homology $H_l(Y_d, E)$ of $E$ \wrt \ $Y_d$
  is a generalized Eulerian $\D_m^\prime$-module for $l = 0, 1$.
 \end{proposition}
\begin{proof}
 We note that the map $E \xrightarrow{Y_d} E$ is $\D_m^\prime$-linear. So it follows that $H_l(Y_d, E)$
 is a $\D_m^\prime$-module for $l = 0, 1$. Let $u \in H_1(Y_d, E)$ be homogeneous. Then as $u \in E_d$. So there exists $a \geq 1$
 such that $(\E_m - |u|)^a = 0$. It follows that  $H_1(Y_d, E)$ is generalized Eulerian. A similar argument yields
 that $H_0(Y_d, E)$ is a generalized Eulerian $\D_m^\prime$-module.
\end{proof}

 An easy induction  (and using \ref{property1})  yields the following result:
\begin{theorem}\label{kos-full-gen-eul}[ with hypotheses as in \ref{koszul-gen-eul}]
  Let $E$ be a generalized Eulerian $\D_m$-module. Then the Koszul homology $H_l(\bY, E)$ of $E$ \wrt \ $\bY = Y_1,\ldots, Y_d$
  is a generalized Eulerian $A_m(K)$-module for $l = 0, 1, \ldots, d$.
 \end{theorem}
 
 \begin{remark}
  We note that the notion of Eulerian $\D_m$-modules is NOT closed under extensions. So even if $E$ is an Eulerian $\D_m$-module
  the Koszul homology  $H_l(\bY, E)$ is only a generalized Eulerian $A_m(K)$-module. 
 \end{remark}

 \s \emph{Koszul homology, de Rham Cohomology of generalized Euelrian $A_m(K)$-modules}
 We recall three earlier results we proved:
 \begin{lemma}\label{sushil}
 Let $M$ be a generalized Eulerian $A_m(K)$-module. Then for $l = 0, 1$
 \begin{enumerate}[\rm (1)]
  \item  $H_l(\partial_m; M)(-1)$ is generalized Eulerian $A_{m-1}(K)$-module; see \cite[3.2]{P2}.
  \item $H_l(X_m, N)$  is generalized Eulerian $A_{m-1}(K)$-module; see \cite[5.3]{PS}.
  \item If $m = 1$ then for $l = 0, 1$,
  \begin{enumerate}[\rm (a)]
  \item 
  $H_l(\partial_1; M)$ is concentrated in degree $-1$, i.e., $H_l(\partial_1; M)_j = 0$ for $j \neq -1$; see \cite[3.5]{P2}.
  \item
  $H_l(X_1; M)$ is concentrated in degree $0$, i.e., $H_l(\partial_1; M)_j = 0$ for $j \neq 0$; see \cite[5.5]{PS}.
  \end{enumerate}
 \end{enumerate}
 \end{lemma}

\section{Graded holonomic modules}
\s \label{setup-graded-holonomic} In this section $K$ is a field of characteristic zero, $A = K[[Y_1,\ldots, Y_d]]$, 
and $R = A[X_1,\ldots, X_m]$ be standard graded with $\deg A = 0$ and $\deg X_i = 1$ for all $i$.
Let
$\Gamma = A<\delta_1,\ldots, \delta_d>$,
 where $\delta_j = \partial/\partial Y_j$, be the ring of $K$-linear differential operators on $A$. 
 Let $\D_m$ be the $m^{th}$ Weyl algebra over $\Gamma$. We note that $\D_m = \Gamma \otimes_K A_m(K)$.
We give a grading on $\D_m$ by giving $\deg \Gamma = 0$, $\deg X_i = 1$ and $\deg \partial_i = -1$. Notice that
$R$ is a graded subring of $\D_m$. 

There is a well-known filtration $\T$ on $\D_m$ such that
$\gr_\T \D_m $ is isomorphic to the polynomial ring over $A$ in $2m + d$-variables
( for instance see \ref{filt}). Furthermore the weak global dimension of $\D_m$ is 
$d + m$, see \cite[3.1.9]{Bjork}. Thus we can define the notion of holonomic $\D_m$-modules. The following result is the main result of this section:

\begin{theorem}\label{mom}
Let $\FF$ be a graded Lyubeznik functor on $ \ ^* Mod(R)$. Then $\FF(R)$ is a graded, generalized Eulerian,  holonomic $\D_m$-module
\end{theorem}

\begin{remark}
To prove Theorem \ref{mom} we need to first give  graded filtration's on relevant objects so that the associated graded filtration has a natural bi-graded structure. This is not done in  the standard reference \cite{Bjork}. So we are compelled to do it carefully in this section.
\end{remark}

The following result is definitely known. We sketch a proof for the convenience of the reader.
\begin{lemma}\label{basis}
 The set 
 $$\{ \delta^{\alpha} \partial^\beta \mid  \text{where} \ \delta^\alpha = \delta_1^{\alpha_1}\cdots\delta_d^{\alpha_d}, \ 
 \text{and} \ \partial^\beta = \partial_1^{\beta_1}\cdots \partial_m^{\beta_m}, \ \ \text{where} 
 \ \alpha_i, \beta_j \geq 0 \} $$
 is a basis of $\D_m$ as a left $R$-module.
 \end{lemma}
\begin{proof}[Sketch of a proof:]
 The result follows since
  $$\{ {X}^{\gamma} \partial^\beta \mid  \text{where} \  X^\gamma = X_1^{\gamma_1}\cdots X_m^{\gamma_m}, \ 
 \text{and} \ \partial^\beta = \partial_1^{\beta_1}\cdots \partial_m^{\beta_m}, \ \ \text{where} 
 \ \gamma_i, \beta_j \geq 0 \} $$
 is a basis of $\D_m$ as a left $\Gamma$-module.  Furthermore
 $$\{ \delta^{\alpha} \mid  \text{where} \ \delta^\alpha = \delta_1^{\alpha_1}\cdots\delta_d^{\alpha_d}, \ 
 \ \alpha_i \geq 0 \} $$
 is a basis of $\Gamma$ as a left $A$-module. Also note that $\delta_j$ commutes with $X_i$ for all $i,j$.
\end{proof}

\s \label{filt} \emph{A filtration of $\D_m$}: \\
We note that $R$ is a graded subring of $\D_m$. Consider the following filtration of $\D_m$ by graded $R$-submodules of $\D_m$.
Set 
\begin{align*}
 \T_0 &= R, \\
 \T_1 &= R + R\delta + R \partial,\\
 \T_n &= \bigoplus_{|\alpha| + |\beta| \leq n} R \delta^\alpha \partial^\beta.
\end{align*}
We note that
\begin{align*}
 \T_n &\subseteq \T_{n+1} \ \ \text{for all} \ n \geq 0.\\
 \bigcup_{n \geq 0} \T_n &= \D_m, \\
 \T_i \T_j &\subseteq \T_{i+j}.
\end{align*}
For convenience set $\T_{-1} = 0$.
Set $\ov{\T_n} = \T_n/\T_{n-1}$ for $n \geq 0$.  We note that $\T_n$ and $\ov{T}_n$ are graded, finitely generated, \emph{free}, $R$-modules.
 
 \s Consider the associated graded ring $\gr_\T \D_m = \bigoplus_{n \geq 0} \ov{\T_n}$. We first note that
 it is a quotient algebra of $R<\ov{\delta_1}, \cdots, \ov{\delta_d}, \ov{\partial_1},\cdots, \ov{\partial_m}>$.
 We now note that as $\delta_j Y_j = Y_j \delta_j + 1$, it follows that $\ov{\delta_j} Y_j = Y_j \ov{\delta_j}$ for all $j$.
 Similarly we get that $X_i$ commutes with $\ov{\partial_i}$. Thus $\gr_\T \D_m$ is a commutative ring and is a quotient
 of a polynomial ring in $d+m$ variables over $R$. It follows from Lemma \ref{basis} that in-fact 
 $$\gr_\T \D_m = R[\ov{\delta_1}, \cdots, \ov{\delta_d}, \ov{\partial_1},\cdots, \ov{\partial_m}], \quad \text{is a polynomial ring}.$$

 \s \label{bi-single-grading}
We note that there are  three graded structures on $\gr_\T \D_m$.
\begin{enumerate}
 \item We give $\deg R = 0$ and $\deg \ov{\delta_i} = 1$ and $\deg \ov{\partial_j} = 1$. We denote $\gr_\T \D_m$ with this graded
 structure as $G^{(0)}(\D_m)$.
 \item
 We give $\deg A = 0$, $\deg X_i = (1,0)$, $\deg \ov{\delta_i} = (0,1)$ and $\deg \ov{\partial_j} = (0,1)$.
 We denote $\gr_\T \D_m$ with this bi-graded
 structure as $G^{(b)}(\D_m)$.
 \item
 We give $\deg A = 0$, $\deg X_i = 1$, $\deg \ov{\delta_i} = 1$ and $\deg \ov{\partial_j} = 1$.
 We denote $\gr_\T \D_m$ with this structure
 structure as $G^{(t)}(\D_m)$.
\end{enumerate}

 \s
Let $M = \bigoplus_{n \in \ZZ}$ be a graded $\D_m$-module. By a \textit{$\T$-compatible filtration} of $M$ we mean a filtration
$\F = \{ \F_n \}_{n \in \ZZ}$ such that
\begin{enumerate}
 \item For all $n \in \ZZ$,   $\F_n$ is an $R$-graded submodule of $M$.
 \item
 $\F_n \subseteq \F_{n+1}$ for all $n \in \ZZ$.
 \item
 $\T_i \F_j \subseteq \F_{i+j}$ for all $i, j$.
 \item
 $\F_n = 0$ for all $n \ll 0$.
 \item
 $\bigcup_{n \in \ZZ} \F_n = M$.
\end{enumerate}
Set $\ov{\F_n} = \F_n/\F_{n-1}$ for all $n \in \ZZ$.
We note that $\gr_\F M = \bigoplus_{n \in \ZZ} \ov{\F_n}$ is a bi-graded $G^{(b)}(\D_m)$-module.
If $\gr_\F M$ is finitely generated as a $G^{(b)}(\D_m)$-module then we say that $\F$ is a good graded filtration on $M$

The following result has a standard proof, see \cite[2.6, 2.7]{Bjork}.
\begin{theorem}\label{fin-gen}
Let $M$ be a graded $\D_m$-module. The following assertions are equivalent:
\begin{enumerate}[\rm (1)]
 \item $M$ is finitely generated as a $\D_m$-module.
 \item There exists a good graded $\T$-compatible filtration on $M$.
\end{enumerate}
\end{theorem}

\s \label{short-dim} \emph{A short note on dimension:} We first note that first  dimension of finitely generated modules over commutative  Noetherian local rings is defined.
In general we have the following:
\begin{enumerate}
\item
Let $S$ be a Noetherian ring and let $E$ be a finitely generated $S$-module. Then
\[
\dim_S E = \max \{ \dim_{S_\m} E_\m \mid \m \ \text{a maximal ideal of } \ S \ \}.
\]
\item
If $S = \bigoplus_{n \geq 0} S_n$ and $E$ is a finitely generated  graded $S$-module then
\[
\dim_S E = \max \{ \dim_{S_\m} E_\m \mid \m \ \text{a  graded maximal ideal of } \ S \ \}, \ \text{see \cite[1.5.8]{BH}}.
\]
Furthermore a graded maximal ideal $\m$ of $S$ is of the form $(\n, S_+)$ where 
$\n$ is a maximal ideal of $S_0$. In particular if $S_0$ is local with maximal ideal $\m_0$ then
\[
\dim_S E = \dim_{S_\mathcal{M}} E_{\mathcal{M}}, \quad \text{where} \ \mathcal{M} = (\m_0, S_+).
\]
\item 
If $S$ is local with maximal ideal $\n$ and $\widehat{S}$ is the completion of $S$ \wrt \ $\n$ then
\[
\dim_S E = \dim_{\widehat{S}} E\otimes_S \widehat{S}.
\]
\end{enumerate}

We will need the following result which is perhaps well-known. Absence of a suitable reference has forced me to include it here:
\begin{proposition}\label{bi-graded-dim}
Let $S = \bigoplus_{i,j \geq 0} S_{ij}$ be a bigraded Noetherian ring and let $E$ be a finitely generated bi-graded $S$-module. Assume $S_{0,0}$ is local with maximal ideal $\n$. Set $\mathcal{M} = (\n , \bigoplus_{i+j >0} S_{ij})$, the maximal bi-graded ideal of $S$. Then
\[
\dim_S E = \dim_{S_\mathcal{M}} E_\mathcal{\M}.
\]
\end{proposition}
\begin{remark}
The reason why Proposition \ref{bi-graded-dim} requires a proof is that the 
extension of \cite[1.5.8]{BH} to the bigraded case is not known (and according to the author is probably wrong).
\end{remark}
\s \label{tot} To prove Proposition \ref{bi-graded-dim} we need the following well-known construction: 

Let $S = \bigoplus_{i,j \geq 0} S_{ij}$ be a bigraded Noetherian ring. Then
$S$ has a $\mathbb{N}$-graded structure as given below:

Set $S^{(t)}_n  = \bigoplus_{i+j = n} S_{ij} $ and
 $S^{(t)} = \bigoplus_{n \geq 0} S^{(t)}_n$.
 We note that $S = S^{(t)}$ as rings and the later is $\mathbb{N}$-graded.
 
 If   $E = \bigoplus_{i,j \in \ZZ} E_{ij}$ is a bi-graded $S$-module then we can 
 give it a $\ZZ$-graded structure over $S^{(t)}$ as follows: 
 
 Set $E^{(t)}_n  = \bigoplus_{i+j = n} E_{ij} $ and
 $E^{(t)} = \bigoplus_{n \in \ZZ} E^{(t)}_n$. We note that $E = E^{(t)}$ and the later is a $\ZZ$-graded $S^{(t)}$-module.
 
 We now give
 \begin{proof}[Proof of Proposition \ref{bi-graded-dim}]
 Give $\mathbb{N}$-graded structure on $S$ as in \ref{tot}. Furthermore give
 $E$ a $\ZZ$-graded structure over $S^{(t)}$ as in \ref{tot}.
 
 Notice that $\mathcal{M}$ is also the unique maximal homogeneous ideal of $S^{(t)}$. The result follows from \ref{short-dim}(2).
 \end{proof}

\s \textit{Dimension of graded $\D_m$-modules:}
We first note the three graded structures on $\gr_T \D_m$. We work with 
$G^{(b)}(\D_m)$, the bi-graded structure on $\gr_T \D_m$. We note that $G^{(b)}(\D_m)_{0,0} = A$ is local with maximal ideal 
$\n = (Y_1,\ldots, Y_d)$. Let $\M$ be the unique bi-graded maximal ideal of $G^{(b)}(\D_m)$.

Then we have
\begin{lemma}\label{local-dim}
Let $E$ be a  finitely generated, graded, $\D_m$-module. Let $\F$ and $\G$ be two graded,  good, $\T$-compatible filtration's of $E$. Then
\begin{enumerate}[\rm (1)]
\item
\[
\dim_{(\gr_\T \D_m)_\M} (\gr_\F E)_\M  =  \dim_{(\gr_\T \D_m)_\M} (\gr_\G E)_\M
\]
\item
$ \dim \gr_\F E = \dim \gr_\G E$ as $\gr_\T \D_m$-modules.
\end{enumerate}
\end{lemma}
\begin{proof}
We note that $\gr_\F E$ and $\gr_\G E$ are finitely generated  bi-graded $G^{(b)}(\D_m)$-modules.

(1) This can be proved along the lines of \cite[2.6.2]{Bjork}.

(2) This follows from  Proposition \ref{bi-graded-dim}.
\end{proof}
 
 \s In view of Lemma \ref{local-dim} we have the following un-ambiguous definition of
 a  finitely generated, graded, $\D_m$-module $E$,
 set
 \[
 \dim E = \dim \gr_\F E   \quad \text{ as a $\gr_\T \D_m$-module}, where
 \]
 $\F$ is a good, graded $\T$-compatible filtration of $E$.
 
 If $E = 0$ or if $\dim E = d + m$ then we say $E$ is a graded \textit{holonomic }$\D_m$-module.
 
 \s \textit{Crucial idea:} The main idea to Prove Theorem \ref{mom} is the following:
 
 Let $\widehat{R} = A[[X_1,\ldots, X_m]] = K[[Y_1,\ldots, Y_d, X_1,\ldots, X_m]]$. Let $\Sc_{m+d} = \widehat{R}<\delta_1,\ldots, \delta_d, \partial_1,\ldots,\partial_m>$ be the ring of $K$-linear differential operators on $\widehat{R}$ where $\delta_j = \partial/\partial Y_j$ and $\partial_i = \partial/\partial X_i$.
  We note that derivations $\delta_j, \partial_i$ on $R$ extend to $\widehat{R}$. We will also denote $\Sc_{m+d}$ by $\Sc$ if $m,d$ are clear by the context. We also note that $\D_m$ can be naturally considered as a sub-ring of $\Sc$.
 
 Let $E$ be a $\D_m$-module. Consider $E^\prime = \widehat{R}\otimes_R E$. Then $E^\prime$ has a natural structure of $\Sc$-module as follows:
 Let $a \in \widehat{R}$ and $e \in E$. Set
 \begin{align*}
 r \cdot (a\otimes e) &=  ra\otimes e, \ \text{here $r \in \widehat{R}$},\\
 \delta_j\cdot (a\otimes e) &=  \delta_j(a)\otimes e  + a \otimes \delta_j e, \\
  \partial_i\cdot (a\otimes e) &=  \partial_i(a)\otimes e  + a \otimes \partial_i e
 \end{align*}
 If $E$ is a finitely generated  $\D_m$-module then it is clear that $E^\prime$ is a finitely generated $\Sc$-module.
 A simple and extremely essential result is that for graded $\D_m$-modules, the converse holds, i.e.,
 \begin{proposition} \label{crucial}
 Let $E$ be a graded $\D_m$-module. If $E^\prime $ is a finitely generated $\Sc$-module then $E$ is a finitely generated $\D_m$-module.
 \end{proposition}
 \s \label{cruc-funct} To prove Proposition \ref{crucial} we first note the following facts:   Let $\M$ be the unique  maximal graded ideal of $R$. Let $\widehat{R_\M} $ be the completion of $R$ \wrt \ $\M$. Then note $\widehat{R_\M} = \widehat{R}$. We also have the following well-known results:
\begin{enumerate}
\item
The functor $-\otimes_R R_\M \colon  \ ^* Mod(R) \rt Mod(R_\M)$ is faithfully flat. 
\item
The functor
 $-\otimes_{R_\M}  \widehat{R}\colon   Mod(R_\M) \rt Mod(\widehat{R})$ is faithfully flat.
 \item
 Thus the composite functor $-\otimes_R \widehat{R} \colon  \ ^* Mod(R) \rt \widehat{R}$ is faithfully flat. 
\end{enumerate} 
 We now give
 \begin{proof}[Proof of Proposition \ref{crucial}]
 Suppose if possible $E$ is not a finitely generated graded $\D_m$-module.
 Then we have a \emph{strictly} increasing chain of \emph{graded} $\D_m$-submodules  of $E$:
 \[
 N_1 \varsubsetneq N_2 \varsubsetneq N_3 \varsubsetneq \cdots \varsubsetneq N_i \varsubsetneq N_{i+1} \varsubsetneq \cdots
 \]
 By \ref{cruc-funct} we get that
 \begin{enumerate}
 \item
 $N_i^\prime = \widehat{R}\otimes_R N_i$ is a $\Sc$-submodule of $E^\prime$.
 \item
 $N_i^\prime \varsubsetneq N_{i+1}^\prime$ for all $i \geq 1$.
 \end{enumerate}
 This contradicts the fact that $E^\prime$ is a finitely generated $\Sc$-module.
 \end{proof}
 
 \s  Consider the $\sum$ filtration  on $\Sc$
where
$$\sum_k = \{ Q \in D \mid Q = \sum_{|\alpha| + |\beta| \leq k} q_{\alpha, \beta}(\bY,\bX)\delta^\alpha\partial^\beta \}, $$ 
is the set of differential operators of order $\leq k$. 
The associated graded ring $\gr_{\sum} D = \sum_0 \oplus \sum_1/\sum_0 \oplus \cdots $ is isomorphic to the polynomial ring $\widehat{R}[\xi_1,\ldots, \xi_d, \zeta_1,\cdots, \zeta_m]$ where $\xi_j, (\zeta_i)$ are the images of $\delta_j, 
(\partial_i)$ in $\sum_1/\sum_0$. 

We note that $\widehat{R}\otimes_R \T_k = \sum_k$ for all $k \geq 0$. More is true:
\begin{proposition}\label{filt-flat}
Let $E$ be a graded $\D_m$-module. Let $\F = \{ \F_n \}_{n \in \ZZ}$ be a $\T$-compatible filtration on $E$. Then
\begin{enumerate}[\rm (1)]
\item
 $\widehat{\F} = \{ \widehat{R}\otimes_R\F_n \}_{n \in \ZZ}$ is a $\sum$-compatible filtration on $E^\prime$.
 \item
$ \widehat{R} \otimes_R \gr_\F E = \gr_{\widehat{\F}} E^\prime$.
\item
If $\F$ is a good $\T$-compatible filtration on $E$ then  $\widehat{\F}$ is a $\sum$-compatible good filtration on $E^\prime$.
\end{enumerate}
\end{proposition}
\begin{proof}
(1) Set $\F_n^\prime = \widehat{R}\otimes_R \F_n$. We note the following:
\begin{enumerate}[\rm (a)]
\item
$\F_n^\prime = 0$ for $n \ll 0$.
\item
$\F_n^\prime \subseteq \F_{n+1}^\prime$ for all $n \in \ZZ$.
\item
$\bigcup_{n \in \ZZ} \F_n^\prime = E^\prime$.
\item
$$\widehat{R} \otimes_R \frac{\F_n}{\F_{n-1}} \cong  \frac{\F_n^\prime}{\F_{n-1}^\prime}  \quad \text{ for all $n \in \ZZ$.} $$
\end{enumerate}
We also note that  for all $n \in \ZZ$,
\begin{align*}
\delta_j(\widehat{R} \otimes_R \F_n) &\subseteq \delta_j(\widehat{R})\otimes \F_n + \widehat{R}\otimes \delta_j \F_n \subseteq \widehat{R} \otimes \F_{n+1} \\
\partial_i(\widehat{R} \otimes_R \F_n) &\subseteq \partial_i(\widehat{R})\otimes \F_n + \widehat{R}\otimes \partial_i \F_n \subseteq \widehat{R} \otimes \F_{n+1}
\end{align*}
It follows that  $\widehat{\F} $ is a $\sum$-compatible filtration on $E^\prime$.

(2) This essentially follows from (d) above.

(3) If we give $\gr_\T \D_m$ grading as in \ref{bi-single-grading}(1) then it follows that 
$$\widehat{R}\otimes_R G^{(0)}(\D_m) \cong \gr_{\sum} \Sc. $$
By (2) it follows that if $\F$ is a good filtration on $E$ then $\widehat{\F}$ is a good filtration on $E^\prime$.
\end{proof}

We now relate dimensions of $E$ and $E^\prime$.
\begin{theorem}\label{compare-dimension}
Let $E$ be a finitely generated graded $\D_m$-module. Then dimension of $E$ as a $\D_m$-module is the same as dimension of $E^\prime$ as a $\Sc$-module.
\end{theorem}
\begin{proof}
Let $\F$ be a good $\T$-compatible graded filtration on $E$. Then by \ref{filt-flat}, 
$\widehat{\F}$ is a $\sum$-compatible good filtration on $E^\prime$.

Let $\M$ be the unique bi-graded maximal bi-homogeneous ideal of $\gr_\T \D_m$.
Then we have
\[
\dim E = \dim_{(\gr_\T \D_m)_\M} (\gr_\F E)_\M
\]
Finally let the completion of $(\gr_\T \D_m)_\M$ \wrt \  $\M$ be $B$. We note 
that
$$B = k[[Y_1,\ldots, Y_d, \ov{\delta_1}, \ldots, \ov{\delta_d}, X_1,\ldots, X_m, \ov{\partial_1}, \ldots, \ov{\partial_m}]].$$
By \ref{short-dim}(3) we have
\begin{equation}\label{Bella}
\dim_{(\gr_\T \D_m)_\M} (\gr_\F E)_\M = \dim_B  \left( \gr_\F  E \otimes_{\gr_\T \D_m} B \right).
\end{equation}

By \ref{filt-flat} $\widehat{\F}$ is a good -filtration on $E^\prime$.
Let $\N$ be the unique  maximal homogeneous ideal of $\gr_{\sum} \Sc$.
Then we have
\[
\dim E^\prime  = \dim_{(\gr_{\sum} \Sc)_\N} (\gr_{\widehat{\F}} E^\prime )_\N.
\]

Note that $\gr_{\sum} \Sc  = \widehat{R} \otimes_R \gr_\T \D_m$. It follows that

\begin{align}\label{rach}
\gr_{\sum } \Sc \otimes_{\gr_\T \D_m} \gr_\F E &=  \widehat{R}\otimes_R \gr_\T \D_m \otimes_{\gr_\T \D_m} \gr_\F E \\
&= \widehat{R}\otimes_{\gr_\T \D_m} \gr_\F E \\
&= \gr_{\widehat{\F}} E^\prime. 
\end{align}
% \\
%&=
We note that $B$ is also the completion of $\gr_{\sum} \Sc$ \wrt \ $\N$.
Now notice by \ref{rach} we have
\begin{align*}
B \otimes_{\gr_{\sum} \Sc} \gr_{\widehat{\F}} E^\prime  &= B \otimes_{\gr_{\sum} \Sc} \gr_{\sum } \Sc \otimes_{\gr_\T \D_m} \gr_\F E \\
 &= B \otimes_{\gr_\T \D_m} \gr_\F E.
\end{align*}
The result now follows from \ref{Bella} and \ref{short-dim}(3).
\end{proof}
An important corollary of Theorem \ref{compare-dimension} is the following:
\begin{corollary}\label{rach2}
Let $E$ be a graded $\D_m$-module. If $E^\prime$ is a holonomic $\Sc$-module then $E$ is a holonomic $\D_m$-module.
\end{corollary}
\begin{proof}
As $E^\prime$ is a holonomic $\Sc$-module, it is in particular a finitely generated $\Sc$-module. So by \ref{crucial} we get that $E$ is a finitely generated $\D_m$-module. By
Theorem \ref{compare-dimension} we get that  dimension of $E$ as a $\D_m$-module is the same as dimension of $E^\prime$ as a $\Sc$-module ( which is $d + m$ as $E^\prime$ is holonomic).

Thus $E$ has dimension $d + m$ as a $\D_m$-module. So it is holonomic.
\end{proof}
As a trivial consequence of the above corollary we get
\begin{proof}[Proof of Theorem  \ref{mom}]
By Theorem \ref{gen-eul-Lyu}, $\FF(R)$ is a graded, generalized Eulerian $\D_m$-module.
By \ref{flat-L} there exists a Lyubeznik functor $G$ on $Mod(\widehat{R})$ such that $G(\widehat{R}) = \widehat{R}\otimes_R \FF(R)$. By \cite[2.2d]{Lyu-1} we get that $G(\widehat{R})$
is a holonmic $\Sc$-module. So by Corollary \ref{rach2} the result follows.
\end{proof}

\s\label{Koszul-M} \emph{Koszul homology:} Let $A^\prime = K[[Y_1,\ldots, Y_{d-1}]]$ and  $R^\prime = A^\prime[X_1,\ldots, X_m]$.
We give the standard grading on $R^\prime$.
Set $\Gamma^\prime = A^\prime<\delta_1,\ldots, \delta_{d-1}>$ the ring of $K$-linear differential operators on $A^\prime$.
Finally set $\D_m^\prime = A_m(\Gamma^\prime)$. Note that $D_m^\prime$ is graded by giving $\deg \Gamma^\prime = 0$, 
$\deg X_i = 1$ and $\deg \partial_i = -1$.
We note that $A^\prime, R^\prime, \Gamma^\prime$ and $\D_m^\prime$ are subrings
of $A, R, \Gamma$ and $\D_m$-respectively. Let $E$ be a $\D_m$-module. Consider the Koszul homology groups of $E$ \wrt \ $Y_d$ 
\begin{align*}
 H_1(Y_d, E) &= \{ e \in E \mid Y_d e = 0 \}, \\
 H_0(Y_d, E) &= E/Y_d E.
\end{align*}
We note that the map $E \xrightarrow{Y_d} E$ is $\D_m^\prime$-linear. It follows that $H_l(Y_d, E)$ is a graded $\D_m^\prime$-module
for $l = 0, 1$. The following result is extremely useful to us:
\begin{theorem}\label{Koszul-local}[with hypotheses as in \ref{Koszul-M}]
 If $E$ is a graded holonomic, generalized Eulerian  $\D_m$-module then  $H_l(Y_d, E)$ are graded holonomic, generalized Eulerian $\D_m^\prime$-modules
for $l = 0, 1$. 
\end{theorem}
\begin{proof}
By Proposition \ref{kos-induct} we get that $H_l(Y_d, E)$ are graded generalized Eulerian $\D_m^\prime$-modules
for $l = 0, 1$. 

 Let $\widehat{R}$ (and $\widehat{R^\prime}$) be completion of $R$ (respectively $R^\prime$)
 \wrt \ its maximal homogeneous ideal. Let $\Sc_{m+d}$ be the ring of $K$-linear differential operators on $\widehat{R}$.
 Finally let $\Sc_{m+d-1}$ be the ring of $K$-linear differential operators on $\widehat{R^\prime}$.
 
 Set $L = \widehat{R}\otimes_R E$. Then by \ref{compare-dimension} we get that $L$ is a holonomic $\Sc_{d + m}$-module.
 So by \cite[3.4.2 and 3.4.4]{Bjork} we get that $H_l(Y_d, L)$ are holonomic $\Sc_{d + m}$-modules.
 We note that for $l = 0, 1$;
 \begin{align*}
  H_l(Y_d, L) &= \widehat{R}\otimes_R H_l(Y_d, E), \\
  &= \widehat{R}\otimes_R R^\prime \otimes_{R^\prime} H_l(Y_d, E), \\
  &= \widehat{R^\prime} \otimes_{R^\prime} H_l(Y_d, E)
 \end{align*}
By \ref{rach2} it follows that $H_l(Y_d, E)$ are graded holonomic $\D_m^\prime$-modules.
\end{proof}
An easy induction to Theorem \ref{Koszul-local} yields
\begin{corollary}\label{full-koszul-hom}[with hypotheses as in \ref{Koszul-M}]
 Let $E$ be a graded holonomic, \\ generalized Eulerian $\D_m$-module. Then for all $\nu \geq 0$ the Koszul homology \\
 $H_\nu(Y_1,\ldots, Y_d; E)$ are graded holonomic, generalized Eulerian  modules over the Weyl algebra $A_m(K)$.
\end{corollary}

\s \label{support-m}\textit{ A graded component supported ONLY at the maximal ideal of $A$:}  Let $M = \bigoplus_{n \in \ZZ}M_n$ be a graded $\D_m$ module.
Suppose $M_c$ is supported ONLY at the maximal ideal $\n = (Y_1, \ldots, Y_d)$ of $A$. Then 
$M_c = E^{\alpha}$ where $E$ is the injective hull of
$K$ as an $A$-module (here $\alpha$ is an ordinal, possibly infinite). To see this note that $\Gamma$ the ring of $K$-lineral differential operators over $A$ is
contained in the degree zero component of $\D_m$. So $M_c$ is a $\Gamma$-module supported at $\n$. The result now follows from
\cite[2.4(a)]{Lyu-1}.

\begin{remark} \label{key-remark}
 Suppose $M = \bigoplus_{n \in \ZZ}M_n$ is a graded holonomic, generalized Eulerian  $\D_m$-module. Suppose $M_c$ is supported ONLY at the maximal
 ideal $\n = (Y_1, \ldots, Y_d)$ of $A$. Then by \ref{full-koszul-hom},  \ref{support-m},\ref{inj-koszul} it follows that
 $H_d(\bY, M)$ is a graded holonomic, generalized Eulerian $A_m(K)$-module with $H_d(\bY, M)_c \neq 0$.
 In fact if $M_c \cong E_A(K)^\alpha$ then $H_d(\bY, M)_c \cong K^\alpha$.
\end{remark}

 \section{Vanishing}

 We first prove the following:
 \begin{theorem}\label{Weyl}
  Let $M = \bigoplus_{n\in \ZZ} M_n$ be a graded holonomic generalized Eulerian $A_m(K)$-module.
  Suppose $M_n = 0$ for all $|n| \gg 0$. Then $M = 0$.
 \end{theorem}
\begin{proof}
We prove the result by induction on $m$. We first consider the case $m = 1$.
Suppose if possible $M \neq 0$. Then there exists $r, s$ such that
 \begin{align*}
  M_r \neq 0 &\quad \text{and} \ M_n = 0 \ \text{for all} \ n \geq r+1, \\
  M_s \neq 0 &\quad \text{and} \ M_n = 0 \ \text{for all} \ n \leq s - 1.
 \end{align*}
\textit{Claim-1} $r = -1$. Suppose if possible $r \neq 0$. Consider the map $M(-1) \xrightarrow{X_r} M$. Looking at its $(r + 1)^{st}$
component  we get $H_1(X_r, M)_{r + 1} = M_r \neq 0$. But as $M$ is a holonomic $A_1(K)$-module, 
we get that $H_1(X_1, M)$ is a finite dimensional $K$-vector space
concentrated in degree zero, see \ref{sushil}. So $r = -1$. We also get
$M_{-1}$ is a finite dimensional $K$-vector-space.

\textit{Claim-2} $s = 0$.  Consider the map $M(+1) \xrightarrow{\partial_1} M$. Looking at its
$s-1$ component we get that $H_1(\partial_1, M)_{s-1} = M_s \neq 0$. By \ref{sushil}, $H_1(\partial_1, M)$ is a finite dimensional
$K$-vector space concentrated in degree $-1$.  So $s = 0$. We also get
$M_{-1}$ is a finite dimensional $K$-vector-space.

Thus  the $A_1(K)$-module $M = M_{-1} \oplus M_0$ is non-zero and has finite dimension as a $K$-vector space. This is a contradiction.
\cite[1.4.2]{Bjork}.

We now assume the result for $m -1$ and prove it for $m$.
Suppose if possible $M \neq 0$. Then there exists $r, s$ such that
 \begin{align*}
  M_r \neq 0 &\quad \text{and} \ M_n = 0 \ \text{for all} \ n \geq r+1, \\
  M_s \neq 0 &\quad \text{and} \ M_n = 0 \ \text{for all} \ n \leq s - 1.
 \end{align*}
 Set $N = H_1(X_m, M)$. By \ref{sushil} we get that $N$ is a generalized Eulerian, holonomic, $A_{m-1}(K)$-module.
 Also note that $N_j = 0$ for all $|j| \gg 0$. Furthermore $N_{r+1} = M_r \neq 0$. This contradicts our induction hypothesis.
\end{proof}

We now state and prove a result which implies Theorem \ref{vanish}.
\begin{theorem}\label{vanish-general}
Let $A$ be a regular ring containing a field of characteristic zero and let $R = A[X_1,\ldots, X_m]$ be standard graded.
Let $\FF$ be a graded Lyubeznik functor on $\ ^* Mod(R)$ and set $M = \FF(R) = \bigoplus_{n \in \ZZ}M_n$. If
$M_n = 0$ for all $|n| \gg 0$ then $M = 0$.
\end{theorem}
\begin{proof}
 Suppose if possible $M_c \neq 0$ for some $c$. Let $P$ be a minimal prime of $M_c$ and let $B = \widehat{A_P}$. Also 
 set $S = B[X_1,\ldots, X_m]$.
 We note that by \ref{std-op} the functor $G(-) = B \otimes_R \FF(-)$ is a graded Lyubeznik functor on $ \ ^* Mod(S)$.
 
 By Cohen-structure theorem $B = K[[Y_1,\ldots, Y_g]]$ where $K = \kappa(P)$ the residue field of $A_P$  and $g = \height_A P$.
 Let $\Lambda $ be the ring of $K$-linear differential operators on $B$ and set $\D_m = A_m(\Lambda)$ the $m^{th}$-Weyl algebra
 over $\Lambda$. Then by Theorem \ref{mom}  we get that $N = G(R)$ is a graded holonmic, generalized Eulerian
 $\D_m$-module. Notice $N_j = B\otimes_A M_j = 0$ for $|j| \gg 0$. Furthermore by \ref{min-loc} we get that 
 $N_c = (M_c)_P \neq 0$. 
 
 As $P$ is the minimal prime of $M_c$ we get that $N_c$ is supported  ONLY at the maximal ideal of $B$.
 By \ref{key-remark} we get that
 $V = H_d(\bY, N)$ is a graded holonomic, generalized Eulerian $A_m(K)$-module with $V_c \neq 0$. 
As $V \subseteq N$ we get 
  $V_j = 0$ for $|j| \gg 0$.  
 This contradicts  Theorem \ref{Weyl}. Thus $M = 0$.
\end{proof}

 \section{Tameness and Rigidity} 
 It is convenient to prove Tameness and rigidity of graded local cohomology modules together.
 We first show
 \begin{theorem}
  \label{rig-tame-dim-0}
  Let $M = \bigoplus_{n \in \ZZ} M_n$ be a holonomic generalized Eulerian $A_m(K)$-module. Then
  \begin{enumerate}[\rm (I)]
   \item The following assertions are equivalent:
   \begin{enumerate}[\rm (a)]
    \item $M_n \neq 0$ for infinitely many  $n < 0$.
    \item There exists $r$ such that $M_n \neq 0$ for all $n \leq r$.
    \item $M_n \neq 0$ for \emph{all} $n \leq -m$.
    \item
    $M_s \neq 0$ for \emph{some} $s \leq -m$
   \end{enumerate}
 \item The following assertions are equivalent:
 \begin{enumerate}[\rm (a)]
  \item $M_n \neq 0$ for infinitely many $n \geq 0$.
  \item There exists $s$ such that $M_n \neq 0$ for all $n \geq s$.
  \item $M_n \neq 0$ for \emph{all} $n \geq 0$.
  \item $M_t \neq 0$ for \emph{some} $t \geq 0$
 \end{enumerate}
\end{enumerate}
 \end{theorem}
\begin{proof}
 (I) Clearly $(c) \implies (b) \implies (a) \implies (d)$. We only have to prove
 $(d) \implies (c)$. This we do by induction on $m$.
 
 We first consider the case $m = 1$.
 We have an exact sequence
 \begin{equation}
  \label{rigid-induct-1}
  0 \rt H_1(\partial_1, M)_j \rt M_{j+1} \rt M_j \rt H_0(\partial_1, M)_j \rt 0.
 \end{equation}
By \ref{sushil} we get that  for $l = 0,1$; $H_l(\partial_1, M)$ is concentrated in degree $-1$.
So by \ref{rigid-induct-1} it follows that 
\[
 M_j \cong M_{-1} \quad \text{for all} \ j \leq -1.
\]
So the result follows.

We now assume $m \geq 2$ and the result is known for $m-1$. We have an exact sequence
 \begin{equation}
  \label{rigid-induct-2}
  0 \rt H_1(\partial_m, M)_j \rt M_{j+1} \rt M_j \rt H_0(\partial_m, M)_j \rt 0.
 \end{equation}

By \ref{sushil} $H_l(\partial_m; M)(-1)$ is generalized Eulerian $A_{m-1}(K)$-module. We consider three cases:\\
\textit{Case 1:} $H_0(\partial_m, M)(-1)_j \neq 0$ for some $j \leq -m + 1$. \\
By the induction hypotheses it follows that $H_0(\partial_m, M)(-1)_j \neq 0$ for all $j \leq -m +1$.
So $H_0(\partial_m, M)_j \neq 0$ for all $j \leq -m$. By exact sequence \ref{rigid-induct-2} it follows that $M_j \neq 0$
for all $j \leq -m$.

\textit{Case 2:} $H_1(\partial_m, M)(-1)_j \neq 0$ for some $j \leq -m + 1$. \\
By the induction hypotheses it follows that $H_1(\partial_m, M)(-1)_j \neq 0$ for all $j \leq -m +1$. So
$H_1(\partial_m, M)_j \neq 0$ for all $j \leq -m$. By exact sequence \ref{rigid-induct-2} it follows that $M_j \neq 0$
for all $j \leq -m + 1$.

\textit{Case 3:} For $l = 0,1$ we have  $H_l(\partial_m, M)(-1)_j = 0$ for ALL $j \leq -m + 1$. \\
Therefore for $l = 0,1$ we have  $H_1(\partial_m, M)_j = 0$ for ALL $j \leq -m$. 
By exact sequence \ref{rigid-induct-2} it follows that
$M_j \cong M_{-m+1}$ for all $j \leq -m +1$. The result follows.

\

(II) Clearly $(c) \implies (b) \implies (a) \implies (d)$. We only have to prove
 $(d) \implies (c)$. This we do by induction on $m$.
 
We first consider the case $m = 1$.
 We have an exact sequence
 \begin{equation}
  \label{rigid-induct-3}
  0 \rt H_1(X_1, M)_j \rt M_{j-1} \rt M_j \rt H_0(X_1, M)_j \rt 0.
 \end{equation}
By \ref{sushil} we get that for $l = 0, 1$; $H_l(X_1, M)$ is concentrated in degree $0$.
So by \ref{rigid-induct-3} it follows that 
\[
 M_j \cong M_{0} \quad \text{for all} \ j \geq 0.
\]
So the result follows.

We now assume $m \geq 2$ and the result is known for $m-1$. We have an exact sequence
 \begin{equation}
  \label{rigid-induct-4}
  0 \rt H_1(X_m, M)_j \rt M_{j-1} \rt M_j \rt H_0(X_m, M)_j \rt 0.
 \end{equation}

By \ref{sushil} $H_l(X_m; M)$ is generalized Eulerian $A_{m-1}(K)$-module for $l = 0,1$. We consider three cases:\\
\textit{Case 1:} $H_0(X_m, M)_j \neq 0$ for some $j \geq 0$. \\
By the induction hypotheses it follows that $H_0(X_m, M)_j \neq 0$ for all $j \geq 0$.
 By exact sequence \ref{rigid-induct-4} it follows that $M_j \neq 0$
for all $j \geq 0$.

\textit{Case 2:} $H_1(X_m, M)_j \neq 0$ for some $j \geq 0$. \\
By the induction hypotheses it follows that $H_1(X_m, M)_j \neq 0$ for all $j \geq  0$. So
by exact sequence \ref{rigid-induct-4} it follows that $M_j \neq 0$
for all $j \geq -1$.

\textit{Case 3:} For $l = 0,1$ we have  $H_l(X_m, M)_j = 0$ for ALL $j \geq 0$. \\
By exact sequence \ref{rigid-induct-4} it follows that
$M_j \cong M_{-1}$ for all $j \geq -1$. The result follows.
\end{proof}

 We now prove the following surprising rigidity theorem
 \begin{theorem}
  \label{rig-tame-suprise-dim-0}
  Let $M = \bigoplus_{n \in \ZZ} M_n$ be a holonomic, generalized Eulerian, $A_m(K)$-module with $m \geq 2$ Then 
   the following assertions are equivalent:
   \begin{enumerate}[\rm (i)]
    \item $M_n \neq 0$ for all  $n \in \ZZ $.
    \item There exists $r$  with $-m < r < 0$ such that $M_r \neq 0$. 
   \end{enumerate}
 \end{theorem}
 \begin{proof}
  We only have to prove $(ii) \implies (i)$. This we do by induction on $m$. 
  
  We first consider the case $m = 2$.  We have $M_{-1} \neq 0$.\\
  \textit{Claim-1:} $M_j \neq 0$ for \textit{infinitely} many $j \leq 0$.\\
  If Claim-1 is false then by Theorem \ref{rig-tame-dim-0} we get that $M_j  = 0$ for all $j \leq -2$.
  We have exact sequence
  \begin{equation}
  \label{rigid-induct-5}
  0 \rt H_1(\partial_2, M)_j \rt M_{j+1} \rt M_j \rt H_0(\partial_2, M)_j \rt 0.
 \end{equation}
 By \ref{sushil} the de Rham homology modules $H_l(\partial_2, M)(-1)$ are generalized Eulerian.
 For $l = 0, 1$ we have $H_l(\partial_2, M)_j = 0$  for $j \leq -3$. So by Theorem  \ref{rig-tame-dim-0} we get 
 $H_l(\partial_2, M)(-1)_j = 0$ for $j \leq -1$ and $l = 0, 1$. So $H_l(\partial_2, M)_j = 0$ for $j \leq -2$.
 By exact sequence \ref{rigid-induct-5} we get for $j = -2$,
 \[
  M_{-1} \cong M_{-2} = 0.
 \]
This contradicts our hypothesis. So our Claim-1 is correct. By Theorem \ref{rig-tame-dim-0}  we get $M_j \neq 0$ for $j \leq -2$.

\textit{Claim-2:} $M_j \neq 0$ for \textit{infinitely} many $j \geq 0$.\\
  If Claim-2 is false then by Theorem \ref{rig-tame-dim-0} we get that $M_j  = 0$ for all $j \geq 0$.
  We have exact sequence
  \begin{equation}
  \label{rigid-induct-6}
  0 \rt H_1(X_2, M)_j \rt M_{j-1} \rt M_j \rt H_0(X_2, M)_j \rt 0.
 \end{equation}
 By \ref{sushil} the Koszul homology modules $H_l(X_2, M)$ are generalized Eulerian.
 For $l = 0, 1$ we have $H_l(X_2, M)_j = 0$  for $j \geq 1$. So by Theorem  \ref{rig-tame-dim-0} we get 
 $H_l(X_2, M)_j = 0$ for $j \geq 0$. 
 By exact sequence \ref{rigid-induct-6} we get for $j = 0$,
 \[
  M_{-1} \cong M_0 = 0.
 \]
This contradicts our hypothesis. So our Claim-2 is correct.
By Theorem \ref{rig-tame-dim-0}  we get $M_j \neq 0$ for $j \geq 0$.

Thus we have proved the result when $m = 2$.

\textit{We now assume that $m \geq 3$ and the result is proved for $m-1$}.

We have $M_r \neq 0$ for some $r$ with $-m < r < 0$. We want to show $M_j \neq 0$ for all $j \in \ZZ$.

 \textit{Claim-3:} $M_j \neq 0$ for \textit{infinitely} many $j \leq 0$.\\
  If Claim-3 is false then by Theorem \ref{rig-tame-dim-0}  we get that $M_j  = 0$ for all $j \leq -m$.
  We have exact sequence
  \begin{equation}
  \label{rigid-induct-7}
  0 \rt H_1(\partial_m, M)_j \rt M_{j+1} \rt M_j \rt H_0(\partial_m, M)_j \rt 0.
 \end{equation}
 By \ref{sushil} the de Rham homology modules $H_l(\partial_m, M)(-1)$ are generalized Eulerian.
 For $l = 0, 1$ we have $H_l(\partial_2, M)_j = 0$  for $j \ll 0$. So by Theorem  \ref{rig-tame-dim-0} and induction hypothesis
 we get 
 $H_l(\partial_m, M)(-1)_j = 0$ for $j \leq -1$ and $l = 0, 1$. So $H_l(\partial_m, M)_j = 0$ for $j \leq -2$.
 By exact sequence \ref{rigid-induct-7} we get,
 \[
  M_{-1} \cong M_{-2} \cong \cdots M_{-m +1} \cong M_{-m} = 0.
 \]
This contradicts our hypothesis. So our Claim-3 is correct. By Theorem \ref{rig-tame-dim-0}  we get $M_j \neq 0$ for 
$j \leq  -m$.

\textit{Claim-4:} $M_j \neq 0$ for \textit{infinitely} many $j \geq 0$.\\
  If Claim-4 is false then by Theorem \ref{rig-tame-dim-0} we get that $M_j  = 0$ for all $j \geq 0$.
  We have exact sequence
  \begin{equation}
  \label{rigid-induct-8}
  0 \rt H_1(X_m, M)_j \rt M_{j-1} \rt M_j \rt H_0(X_m, M)_j \rt 0.
 \end{equation}
 By \ref{sushil} the Koszul homology modules $H_l(X_m, M)$ are generalized Eulerian.
 For $l = 0, 1$ we have $H_l(X_m, M)_j = 0$  for $j \gg 0$. So by Theorem  \ref{rig-tame-dim-0} and induction hypotheses
 we get
 $H_l(X_m, M)_j = 0$ for $j \geq -m + 2$. 
 By exact sequence \ref{rigid-induct-8} we get,
 \[
  M_{-m + 1} \cong M_{-m + 2} \cong \cdots \cong M_{-1} \cong M_0 = 0.
 \]
This contradicts our hypothesis. So our Claim-4 is correct.
By Theorem \ref{rig-tame-dim-0}  we get $M_j \neq 0$ for $j \geq 0$.

As $m \geq 3$ we also have to prove that if $c \neq r$ and $-m < c < 0$ then $M_c \neq 0$. \\
Suppose if possible $M_c = 0$. We have to consider two cases:\\
\textit{Case-1:} $c < r$. \\
By exact sequence \ref{rigid-induct-8} we get  $H_1(X_m, M)_{c+1} = 0$. 
 We recall that  \\ $H_1(X_m, M)$ is generalized Eulerian $A_{m-1}(K)$-module. 
 We have  \\ $H_1(X_m, M)_{c+1} = 0$ and $-m +1   < c + 1 < 0$ (as $c < r < 0$). 
  So by induction hypothesis $H_1(X_m, M)_{j} = 0$
 for $-m+1 < j < 0$.
 
 By exact sequence  \ref{rigid-induct-8} we also  get  $H_0(X_m, M)_{c} = 0$.  We recall that 
 \\ $H_0(X_m, M)$ is generalized Eulerian $A_{m-1}(K)$-module. 
 We have to consider two sub-cases:\\
 \textit{Sub-case 1.1:}  $-m+1 < c < 0$. \\
 By induction hypothesis we have $H_0(X_m, M)_j = 0 $ for $-m+1 < j < 0$.
 Thus by \ref{rigid-induct-8} we have
 \[
  M_{-1} \cong M_{-2} \cong \cdots \cong M_{-m + 2} \cong M_{-m + 1}
 \]
This implies $M_r \cong M_c = 0$, a contradiction.

\textit{Sub-case 1.2:}  $c = -m+1 $. \\
So we have $H_0(X_m, M)_{-m +1} = 0$. So again by induction hypothesis  and Theorem \ref{rig-tame-dim-0}, we get 
$H_0(X_m, M)_j = 0 $ for $j < 0$. In particular $H_0(X_m, M)_j = 0 $ for $-m+1 < j < 0$. So again  by an argument similar to
Sub-case-1.1, we get $M_r \cong M_c = 0$, a contradiction.

\textit{Case-2:} $c > r$. \\
 By exact sequence \ref{rigid-induct-7} we get  $H_1(\partial_m, M)_{c-1} = 0$. 
 We recall that  \\ $H_1(\partial_m, M)(-1)$ is a generalized Eulerian $A_{m-1}(K)$-module. 
 We have  \\ $H_1(\partial_m, M)(-1)_{c} = 0$. As $c > r$ we also have $-m + 1 < c < 0$.
  So by induction hypothesis $H_1(\partial_m, M)(-1)_{j} = 0$
 for $-m+1 < j < 0$. Thus $H_1(\partial_m, M)_j = 0$ for $ -m < j < -1$.
 
 By exact sequence  \ref{rigid-induct-7} we also  get  $H_0(\partial_m, M)_{c} = 0$.  We recall that 
 \\ $H_0(\partial_m, M)(-1)$ is  a generalized Eulerian $A_{m-1}(K)$-module. 
 We have to consider two sub-cases:\\
 \textit{Sub-case 2.1:}  $c \neq -1$. \\
 Then $-m + 2 < c+ 1 < 0$ and $H_0(\partial_m, M)(-1)_{c+1} = 0$. So $H_0(\partial_m, M)(-1)_{j} = 0$
 for $-m+1 < j < 0$. Thus $H_0(\partial_m, M)_j = 0$ for $ -m < j < -1$.
So by \ref{rigid-induct-7} we get
\[
M_{-1} \cong M_{-2} \cong \cdots M_{-m+2} \cong M_{-m+1}.
\]
This implies $M_r \cong M_c = 0$, a contradiction.

\textit{Sub-case 2.2:}  $c = -1 $. \\
So we have $H_0(\partial_m, M)(-1)_0 = 0$. So again by induction hypothesis  and Theorem \ref{rig-tame-dim-0}, we get 
$H_0(\partial_m, M)(-1)_j = 0 $ for $j \geq - m + 2$. In particular $H_0(\partial_m, M)(-1)_j = 0 $ for $-m+1 < j < 0$.
Thus $H_0(\partial_m, M)_j = 0$ for $ -m < j < -1$.
 So again  by an argument similar to
Sub-case-2.1, we get $M_r \cong M_c = 0$, a contradiction.
\end{proof}

We now state and prove a result which implies Theorem \ref{tame} and \ref{rigid}.
\begin{theorem}\label{tame-rigid-general}
Let $A$ be a regular ring containing a field of characteristic zero and let $R = A[X_1,\ldots, X_m]$ be standard graded.
Let $\FF$ be a graded Lyubeznik functor on $\ ^* Mod(R)$ and set $M = \FF(R) = \bigoplus_{n \in \ZZ}M_n$. 
Then
  \begin{enumerate}[\rm (I)]
   \item The following assertions are equivalent:
   \begin{enumerate}[\rm (a)]
    \item $M_n \neq 0$ for infinitely many  $n < 0$.
    \item There exists $r$ such that $M_n \neq 0$ for all $n \leq r$.
    \item $M_n \neq 0$ for \emph{all} $n \leq -m$.
    \item
    $M_s \neq 0$ for \emph{some} $s \leq -m$
   \end{enumerate}
 \item The following assertions are equivalent:
 \begin{enumerate}[\rm (a)]
  \item $M_n \neq 0$ for infinitely many $n \geq 0$.
  \item There exists $s$ such that $M_n \neq 0$ for all $n \geq s$.
  \item $M_n \neq 0$ for \emph{all} $n \geq 0$.
  \item $M_t \neq 0$ for \emph{some} $t \geq 0$
 \end{enumerate}
 \item ($ m \geq 2$:) The following assertions are equivalent:
 \begin{enumerate}[\rm (a)]
    \item $M_n \neq 0$ for all  $n \in \ZZ $.
    \item There exists $r$  with $-m < r < 0$ such that $M_r \neq 0$. 
   \end{enumerate}
\end{enumerate}
\end{theorem}
\begin{proof}
(I) Clearly $(c) \implies (b) \implies (a) \implies (d)$. We only have to prove
 $(d) \implies (c)$.
 Suppose if possible $M_s \neq 0$ for some $s \leq -m$. Let $P$ be a minimal prime of $M_s$ and let $B = \widehat{A_P}$. Also 
 set $S = B[X_1,\ldots, X_m]$.
 We note that by \ref{std-op} the functor $G(-) = B \otimes_R \FF(-)$ is a graded Lyubeznik functor on $ \ ^* Mod(S)$.
 
 By Cohen-structure theorem $B = K[[Y_1,\ldots, Y_g]]$ where $K = \kappa(P)$ the residue field of $A_P$  and $g = \height_A P$.
 Let $\Lambda $ be the ring of $K$-linear differential operators on $B$ and set $\D_m = A_m(\Lambda)$ the $m^{th}$-Weyl algebra
 over $\Lambda$. Then by  Theorem \ref{mom} we get that $N = G(R)$ is a graded holonmic, generalized Eulerian
 $\D_m$-module.  Furthermore by \ref{min-loc} we get that 
 $N_{s} = (M_{s})_P \neq 0$. 
 
 As $P$ is the minimal prime of $M_{s}$ we get that $N_s$ is supported  ONLY at the maximal ideal of $B$.
 By \ref{key-remark} we get that
 $V = H_d(\bY, N)$ is a graded holonomic, generalized Eulerian $A_m(K)$-module with $V_{s} \neq 0$. 
 By Theorem \ref{rig-tame-dim-0} we get  $V_n \neq 0$ for all $n \leq -m$.
 We note that $H_d(\bY, N)_n \subseteq N_n$ for all $n \in \ZZ$. So $N_n \neq 0$ for all 
 $n \leq -m$. As $N_n = M_n \otimes_A B$ it follows that $M_n \neq 0$ for all 
 $n \leq -m$.
 
 (II) and (III) The proof of these assertions are similar to (I). We simply localize at a minimal prime $P$ of $M_i$  where $i$ is appropriately  chosen.
 We then localize and complete $A$ at $P$. Then we take appropriate Koszul homology to reduce the case when $A = K$. The result then follows by Theorem \ref{rig-tame-dim-0} and Theorem \ref{rig-tame-suprise-dim-0}.
\end{proof}

\section{Examples and Proof of Theorem \ref{type2}}
In this section we give examples illustrating Theorem \ref{tame-rigid-general}. We also prove Theorem \ref{type2}.
Throughout $A$ is a regular ring containing a field of characteristic zero 
Also let $R = A[X_1, \ldots, X_m]$ be standard graded.

\begin{example}\label{ex-1}
Let $A$ be any. Let $R_+ = (X_1,\ldots, X_m)$. Then 
\[
 H^m_{R_+}(R)_j = \begin{cases}
 0 & \text{for} \ j \geq -m + 1 \\
 \neq 0  & \text{for} \ j \leq -m.
\end{cases}
\]
\end{example}
At the other extreme we have:
\begin{example}\label{ex-2.1}
Assume $\dim A > 0$ and let $P$ be a prime ideal in $A$ of height $g$. Let 
$I = P R$. Then $H^g_{I}(R) = H^g_P(A)[X_1,\ldots, X_m]$.
In particular
\[
 H^g_I(R)_j = \begin{cases}
 \neq 0 & \text{for} \ j \geq 0\\
 = 0  & \text{for} \ j \leq -1.
\end{cases}
\]
\end{example}

A trivial example of ideal $J$ which is not of the form $QR$ for some ideal $Q$ in $A$ such that $H^i_J(R)$ is supported at non-negative integers is when radical of $J$ in $R$ equals that of $QR$. The following example is different:
\begin{example}\label{ex-2.2}
Let $d \geq 2$
Let $A = K[[Y_1,\ldots, Y_d]]$ or $A = K[Y_1,\dots, Y_d]$. Let $I = (Y_1Y_2, Y_1X_1)R$. We note that there is an exact sequence
\begin{equation}\label{la-la-2.2}
0 \rt H^1_{(Y_1)A}(A) \oplus H^1_{(Y_2)A}(A)  \rt H^1_{(Y_1Y_2)A}(A)  \rt H^1_{(Y_1, Y_2)A}(A)  \rt 0.
\end{equation}
This induces an exact sequence
\[
0 \rt H^1_{(Y_1)R}(R) \oplus H^1_{(Y_2)R}(R)  \rt H^1_{(Y_1Y_2)R}(R)  \rt H^2_{(Y_1, Y_2)R}(R)  \rt 0.
\]
We also have an exact sequence
\[
0 \rt H^1_I(R) \rt H^1_{(Y_1Y_2)R}(R) \xrightarrow{\phi} \left(H^2_{(Y_1Y_2)R}(R)\right)_{Y_1X_1},
\]
where $\phi$ is the natural localization map.
Choose a non-zero  element $u$ in degree zero in $H^1_{(Y_1Y_2)R}(R)$  with the property $Y_1u = 0$, (this is possible by \ref{la-la-2.2}). Then note 
$\phi(u) = 0$. Thus $H^1_I(R)_0 \neq 0$. Also it is clear that $H^1_I(R)$ is supported at non-negative integers. Finally note that $\sqrt{I} \neq \sqrt{QR}$ for any ideal $Q$ of $A$.
\end{example}
The following example is essentially  from \cite[15.1.8]{BS}.
\begin{example}\label{ex-3}
 Assume $m \geq 2$. Let $I = (X_1)$. Then $H^1_I(R)_n$ is an infinitely generated free $A$-module for all $n \in \ZZ$.
\end{example}
Finally we have the following:
\begin{example}\label{finnal-ex-peculair}
Assume $A$ is a domain  with a prime ideal $Q$ of height $m+1$ and an element $\xi$ such that
$(Q,\xi) = A$. An explicit example will be $A = K[Y_1,\ldots, Y_d)$ with $d \geq m+ 1$, and $Q = (Y_1,\ldots, Y_{m+1})$ and $\xi = Y_1  + 1$. We note
$H^{m+1}_Q(A) \neq 0$. So $H^{m+1}_{QR}(R) =  H^{m+1}_Q(A)\otimes_A R \neq 0$.
Recall that $H^m_{R_+}(R)_n $ is  finitely generated  free $A$-module for $n \leq -m$ and zero for $n \geq -m + 1$. Thus we have an exact sequence
\[
0 \rt H^m_{R_+}(R) \xrightarrow{\phi} \left(H^m_{R_+}(R)\right)_\xi \rt H^{m+1}_{(\xi, R_+)}(R) \rt 0,
\]
where $\phi$ is the natural localization map. It follows that
\[
 H^{m+1}_{(\xi, R_+)}(R)_j = \begin{cases}
 0 & \text{for} \ j \geq -m + 1 \\
 \text{non-zero, finitely many copies  of $A_\xi/A$} \  & \text{for} \ j \leq -m.
\end{cases}
\]
We note that the ideals $QR$ and $(\xi, R_+)$ are co-maximal in $R$.
Set $I = QR(\xi, R_+)$. Then it is well-known that
\[
H^{m+1}_I(R) \cong H^{m+1}_{QR}(R) \oplus H^{m+1}_{(\xi, R_+)}(R).
\]
Thus 
\[
 H^{m+1}_{I}(R)_j = \begin{cases}
 0 & \text{for} \  -m < j < 0\\
 \text{non-zero}  &  \ \text{otherwise} 
\end{cases}
\]

\end{example}

Before proving Theorem \ref{type2} we need the following well-known notion:
\begin{definition}
Let $f \in A[X_1,\ldots, X_m]$ (not necessarily homogeneous). \\ By $\content(f)$ we mean the ideal in $A$ generated by coefficients of $A$
\end{definition}
We now give proof of Theorem \ref{type2}. We restate it for the convenience of the reader. We will make the assumption $A$ is a domain to avoid trivial exceptions.
\begin{theorem}\label{type2-pr}
(with hypotheses as in \ref{std}). Further assume $A$ is a domain.
Suppose $J$ is a proper homogeneous ideal in $R$ such that $H^i_J(R)_n \neq 0$ for all $n \geq 0$ and $H^i_J(R)_n = 0$ for all $n < 0$. Then there exists a proper ideal $Q$ of $A$ such that $J \subseteq QR$.
\end{theorem}
\begin{proof}
We  prove the  result by proving the following steps:

\textit{Step-1:} For every $f \in J$, not-necessarily homogeneous, $\content(f)$ is a proper ideal in $A$.

\textit{Step-2:} There exists a proper ideal $Q$ in $A$ such that $\content(f) \in Q$ for ALL $f \in J$.

We note that by Step 2 we have $J \subseteq QR$, which proves our result. 

Next we show:\\
\textit{Step-1} $\implies$ \textit{Step-2}: \\
Consider:
\[
\mathcal{C} = \{ \content(f) \colon f \in J \}.
\] 
$\mathcal{C}$ is a collection of ideals in a Noetherian ring $A$. So $\mathcal{C}$ 
has maximal elements. \\
 \textit{Claim:} $\mathcal{C}$ has a unique maximal element:\\
 Let $\content(g)$ and $\content(h)$ be two maximal elements in $\mathcal{C}$. Set
 $$p = g  +  X_1^{c+1}h \quad \text{where} \ c = \ \text{total degree of} \ g.$$
 Clearly $p \in I$ and $\content(p) \supseteq \content(g) + \content(h)$. By maximality we have
 \[
 \content(g) = \content(h) = \content(p).
 \]
 Thus our claim is proved.
 
 Let $Q$ be the unique maximal element in $\mathcal{C}$. Then clearly $\content(f) \subseteq Q$ for all 
 $f \in J$. This proves Step 2  assuming Step 1 is true:
 
 We now give proof of Step-1:\\ Suppose if possible there exists $f \in J$ such that
 $\content(f) = A$.   Here $f$ \textit{need not} be homogeneous.
 
 We now do our standard procedure:
 Let $P$ be a minimal prime of $M_0$ and let $B = \widehat{A_P}$. Also 
 set $S = B[X_1,\ldots, X_m]$.
  Set $N = H^i_{JS}(S) \cong H^i_J(R)\otimes_A B$.
 
 By Cohen-structure theorem $B = K[[Y_1,\ldots, Y_g]]$ where $K = \kappa(P)$ the residue field of $A_P$  and $g = \height_A P$.
 Let $\Lambda $ be the ring of $K$-linear differential operators on $B$ and set $\D_m = A_m(\Lambda)$ the $m^{th}$-Weyl algebra
 over $\Lambda$. Then by  Theorem \ref{mom} we get that $N$ is a graded holonmic, generalized Eulerian
 $\D_m$-module.  Furthermore by \ref{min-loc} we get that 
 $N_{0} = (M_{0})_P \neq 0$.  Also note that $N_j = 0$ for $j < 0$.
 
 As $P$ is the minimal prime of $M_{0}$ we get that $N_0$ is supported  ONLY at the maximal ideal of $B$.
 By \ref{key-remark} we get that
 $V = H_d(\bY, N)$ is a graded holonomic, generalized Eulerian $A_m(K)$-module with $V_{0} \neq 0$. 
 We note that $H_d(\bY, N)_n \subseteq N_n$ for all $n \in \ZZ$. So $V_n = 0$ for $n < 0$.
 
 Let $f^* $ be the image of $f$ in $B$. Notice $\content(f^*) = B$. Let $\ov{f}$ is the image of $f^*$ in $B/(\bY)[X_1,\ldots, X_m] = K[X_1,\ldots, X_m] = T$. We note that $\ov{f} \neq 0$. Also as $N$ is $(f^*)$-torsion we get that $V$ is $(\ov{f})$-torsion.
 
Let $\D = A_m(K)$ the $m^{th}$-Weyl algebra over $K$.
 Choose $u \neq 0 $ with $ u \in V_0$. Consider the $\D$-linear map $\psi \colon \D \rt V$ which maps $1$ to $u$. Clearly $\psi(\D\partial) = 0$. Thus $\psi$-factors 
 through a $\D$-linear map $\ov{\psi} \colon T \rt V$ which is non-zero. As $T$ is a simple $\D$-module we get that
  $T$ is a $\D$-submodule of $V$. But $V$ is $(\ov{f})$-torsion. Therefore $T$ is $(\ov{f})$-torsion. This is a contradiction.
 \end{proof}

We end this section with two questions.
\begin{question}\label{local-ex4}
Does there exists a  regular local ring $A$  and an ideal $L$  in $R = A[X_1,\ldots, X_m]$ such that 
$H^i_L(R)_n \neq 0$ for all $n \leq -m$, $H^i_L(R)_n \neq 0$ for all $n \geq 0$ and $H^i_L(R) = 0$ for all $n$ with $-m < n < 0$.
\end{question} 
 In Example \ref{finnal-ex-peculair} we had to assume $\dim A \geq m + 1$. So our final question is
 \begin{question}\label{global-ex4}
Does there exists a  regular  ring $A$ of dimension $\leq m$ and an ideal $J$  in $R = A[X_1,\ldots, X_m]$ such that 
$H^i_J(R)_n \neq 0$ for all $n \leq -m$, $H^i_J(R)_n \neq 0$ for all $n \geq 0$ and $H^i_J(R) = 0$ for all $n$ with $-m < n < 0$.
\end{question} 
\begin{remark}
I think that both Questions \ref{local-ex4} and \ref{global-ex4} have a negative answer. However I have no idea how to prove it.
\end{remark}

\section{Infinite generation}
In this section we prove Theorem \ref{inf-gen}. We will make the assumption $A$ is a domain to avoid trivial exceptions. For the convenience of reader we restate Theorem \ref{inf-gen}
\begin{theorem}\label{inf-gen-proof}(with hypotheses as in \ref{std}). Further assume $A$ is a domain. Assume $I \cap A \neq 0$. If $H^i_I(R)_c \neq 0$ then
$H^i_I(R)_c$ is NOT finitely generated as an $A$-module.
\end{theorem}
\begin{proof}
Set $Q = I \cap A$. Suppose if possible $ L = H^i_I(R)_c$ is a non-zero finitely generated $A$-module.

 Let $\m$ be a maximal prime  ideal in $A$ belonging to the support of $L$. Notice as $L$ is $Q$-torsion we have $\m \supseteq Q$. Let $B = \widehat{A_\m}$ the completion of $A$ \wrt \ $\m$. We note that the image of $Q$ in $B$ is non-zero.
 Set $J = QB$.
 
 Set $S = B[X_1,\ldots, X_m]$. We note that $H^i_I(R)\otimes_A B \cong H^i_{IS}(S)$. Set $V = L\otimes_A B = H^i_{IS}(S)_c$. Note $V$ is a finitely generated $B$-module.

 By Cohen-structure theorem $B = K[[Y_1,\ldots, Y_g]]$ where $K = \kappa(\m)$ the residue field of $A_\m$  and $g = \height_A \m$.
 Let $\Lambda $ be the ring of $K$-linear differential operators on $B$ and set $\D_m = A_m(\Lambda)$, the $m^{th}$-Weyl algebra
 over $\Lambda$. Then by  Theorem \ref{mom} we get that $N = H^i_{IB}(S) $ is a graded holonmic, generalized Eulerian
 $\D_m$-module.  In particular $V = N_c$ is a $\Lambda$-module.
 
 Let $V$ be generated as a $R$-module by $v_1,\cdots, v_l$. Each $v_j$ is killed by a power of $J$. It follows that there exists $n$ such that $J^n V = 0$. 
 Let $\n$ be the maximal ideal of $B$. Choose $p \in \n$ of smallest $\n$-order $s$ such that $p V = 0$. We note that for all $i = 1, \ldots, d$,
 $$\delta_i p = p \delta_i  + \delta_i(p),$$
 holds in $\Lambda$.  Notice $p \delta_i V = 0$ since $\delta_i V \subseteq V$ (as $V$ is a $\Lambda$-module).  So $\delta_i(p) V = 0$.
 
 We note that if $s \geq 1$ then some $\partial_i(p)$ will have $\n$-order less than $s$. It follows $s = 0$. Thus $p$ is a unit. So $V = 0$, a contradiction as we were assuming $V$ to be non-zero.
 Thus our assumption is incorrect. Therefore $ H^i_I(R)_c$ is NOT finitely generated as an $A$-module.
\end{proof}

\section{Bass numbers}
 \s \label{set-bass} \textit{Setup:} Let $A$ be a regular ring containing a field of characteristic zero. Let $R = A[X_1,\ldots, X_m]$ be standard graded. Let $\FF$ be a graded Lyubeznik functor on $\ ^* Mod(R)$ and set $M = \FF(R) = \bigoplus_{n \in \ZZ}M_n$. 

By example \ref{ex-3} it is possible that $\mu_i(P, M_n)$  the $i^{th}$-Bass number of $M_n$ \wrt \ $P$ can be infinite for some prime ideal $P$ of $A$. 
Surprisingly  we have the following dichotomy:
\begin{theorem}
\label{bass-basic-proof}(with hypotheses as in \ref{set-bass}).  Let $P$ be a prime ideal in $A$. Fix $j \geq 0$. EXACTLY one of the following hold:
\begin{enumerate}[\rm(i)]
\item
$\mu_j(P, M_n)$ is infinite for all $n \in \ZZ$.
\item
$\mu_j(P, M_n)$ is finite for all $n \in \ZZ$. In this case EXACTLY one of the following holds:
\begin{enumerate}[\rm (a)]
\item
$\mu_j(P, M_n) = 0$ for all $n \in \ZZ$.
\item
$\mu_j(P, M_n) \neq 0$ for all $n \in \ZZ$.
\item
$\mu_j(P, M_n) \neq 0$ for all $n  \geq 0$ and $\mu_j(P, M_n) = 0$ for all 
$n < 0$.
\item
$\mu_j(P, M_n) \neq 0$ for all $n  \leq -m$ and $\mu_j(P, M_n) = 0$ for all 
$n > -m$.
\item
$\mu_j(P, M_n) \neq 0$ for all $n  \leq -m$, $\mu_j(P, M_n) \neq 0$ for all $n  \geq 0$ and $\mu_j(P, M_n) = 0$ for all $n$ with $-m < n< 0$.
\end{enumerate}
\end{enumerate}
\end{theorem}

We will need the following Lemma from \cite[1.4]{Lyu-1}.
\begin{lemma}\label{lyu-lemma}
Let $B$ be a Noetherian ring and let $N$ be a $B$-module ($N$ need not be finitely generated).
Let $P$ be a prime ideal in $B$. If $(H^j_P(N))_P$ is injective for all $j \geq 0$ then
$\mu_j(P,N) = \mu_0(P,H^j_P(N))$ for all $j \geq 0$.
\end{lemma}
The following result shows that the hypothesis of Lemma \ref{lyu-lemma} is satisfied in our case:
\begin{proposition}\label{lyu-lemma-hypoth}
(with hypotheses as in \ref{set-bass}).  Let $P$ be a prime ideal in $A$.  Let $N = M_c$ and let $P$ be a prime ideal in $A$. Then $(H^j_P(N))_P$ is injective for all $j \geq 0$
\end{proposition}
\begin{proof}
Fix $j \geq 0$.
We note that $H^j_{PR}\circ\FF$ is a graded Lyubeznik functor on $\ ^*Mod(R)$. Also notice $H^j_{P}(N) = (H^j_{PR}(\FF(R)))_c$. Finally note that either 
$H^j_P(N)_P = 0$ or $P$ is a minimal prime of $H^j_P(N)$. 

We have nothing to show if $H^j_P(N)_P = 0$. So assume  $H^j_P(N) \neq 0$. 
Let $B = \widehat{A_P}$. Also 
 set $S = B[X_1,\ldots, X_m]$.
 We note that by \ref{std-op} the functor $G(-) = B \otimes_A H^j_{PR}\circ\FF(-)$ is a graded Lyubeznik functor on $ \ ^* Mod(S)$.
 
 By Cohen-structure theorem $B = K[[Y_1,\ldots, Y_g]]$ where $K = \kappa(P)$ the residue field of $A_P$  and $g = \height_A P$.
 Let $\Lambda $ be the ring of $K$-linear differential operators on $B$ and set $\D_m = A_m(\Lambda)$ the $m^{th}$-Weyl algebra
 over $\Lambda$. Then by  Theorem \ref{mom} we get that $L = G(R)$ is a graded holonomic, generalized Eulerian
 $\D_m$-module. Notice  by \ref{min-loc} we get that 
 $L_c = N_P \neq 0$. 
 
 As $P$ is the minimal prime of $M_c$ we get that $L_c$ is supported  ONLY at the maximal ideal of $B$. By \ref{support-m} we get that $L_c = E_B(K)^\alpha$ where 
 $E_B(K)$ is the injective hull of $K$ as a $B$-module (and $\alpha$ some ordinal possibly infinite). 
 But we have
 \[
 E_B(K) \cong E_{A_P}(\kappa(P)) \cong E_A(A/P) \quad \text{as $A$-modules.}
 \]
Thus $(H^j_P(N))_P$  is an injective $A$-module.
\end{proof}
The following result is an essential ingredient to prove Theorem \ref{bass-basic-proof}.
\begin{proposition}\label{bass-dim0}
Let $M = \bigoplus_{n \in \ZZ}M_n$ be a graded holonomic  $ A_m(K)$-module. If \\ $\dim_K M_c < \infty$ for some $c \in \ZZ$ then $\dim_K M_n < \infty$ for all $n \in \ZZ$.
\end{proposition}
\begin{proof}
We prove it by induction on $m$.
We first assume $m = 1$. We have an exact sequence
\begin{equation}\label{bass-eqn}
0 \rt H_1(\partial_1, M)_j \rt M_{j+1} \rt M_j \rt H_0(\partial_1, M)_j \rt 0. 
\end{equation}
As $M$ is holonomic we have that $\dim_K H_l(\partial_1, M) < \infty $ for $l = 0,1$. In particular $\dim_K H_l(\partial_1, M)_j < \infty $ for all $j \in \ZZ$ and $l = 0,1$.
By \ref{bass-eqn} we get that $\dim_K M_{c+1} < \infty$. Iterating we get $\dim_K M_j < \infty$ for all $j \geq c$. Again by \ref{bass-eqn} we get that $\dim_K M_{c-1} < \infty$. Iterating we get $\dim_K M_j < \infty$ for all $j \leq c$. The result follows.

We now assume that $m \geq 2$ and the result is known for $m-1$. We have an exact sequence
\begin{equation}\label{bass-eqn-2}
0 \rt H_1(\partial_m, M)_j \rt M_{j+1} \rt M_j \rt H_0(\partial_m, M)_j \rt 0. 
\end{equation}
As $M$ is holonomic we have that $ H_l(\partial_m, M) $ is holonomic $A_{m-1}(K)$-module for $l = 0,1$. By \ref{bass-eqn-2} we get that $\dim_K H_1(\partial_m, M)_{c-1} < \infty $ and $\dim_K H_0(\partial_m, M)_{c} < \infty $. By induction hypothesis we get that  $ \dim_K H_l(\partial_m, M)_j < \infty $ for all $j \in \ZZ$ and $l = 0,1$. By \ref{bass-eqn-2} we get that $\dim_K M_{c+1} < \infty$. Iterating we get $\dim_K M_j < \infty$ for all $j \geq c$. Again by \ref{bass-eqn-2} we get that $\dim_K M_{c-1} < \infty$. Iterating we get $\dim_K M_j < \infty$ for all $j \leq c$. The result follows.
\end{proof}

We now give
\begin{proof}[Proof of Theorem \ref{bass-basic-proof}]
Let $P$ be a prime ideal in $A$. Fix $j \geq 0$. Suppose if possible $\mu_j(P, M_c) < \infty$ for some $c \in \ZZ$. We show that  $\mu_j(P, M_n) < \infty$ for all $n \in \ZZ$.

By Lemma \ref{lyu-lemma-hypoth} and Proposition \ref{lyu-lemma} we get 
that $\mu_j(P, M_n) = \mu_0(P, H^j_P(M_n))$ for all $n \in \ZZ$.
We note that $(H^j_{PR}\circ \FF)(R)_n = H^j_P(M_n)$ for all $n \in \ZZ$. Furthermore $H^j_{PR}\circ\FF$ is a graded Lyubeznik functor on $\ ^*Mod(R)$.

Let $B = \widehat{A_P}$. Also 
 set $S = B[X_1,\ldots, X_m]$.
 We note that by \ref{std-op} the functor $G(-) = B \otimes_A H^j_{PR}\circ\FF(-)$ is a graded Lyubeznik functor on $ \ ^* Mod(S)$.
 
 By Cohen-structure theorem $B = K[[Y_1,\ldots, Y_g]]$ where $K = \kappa(P)$ the residue field of $A_P$  and $g = \height_A P$.
 Let $\Lambda $ be the ring of $K$-linear differential operators on $B$ and set $\D_m = A_m(\Lambda)$ the $m^{th}$-Weyl algebra
 over $\Lambda$. Then by  Theorem \ref{mom} we get that $L = G(R)$ is a graded holonmic, generalized Eulerian
 $\D_m$-module. Notice  by \ref{min-loc} we get that 
 $L_n = (H^j_P(M_n))_P $ for all $n \in \ZZ$.
 
 Note that either $(H^j_P(M_n))_P  = 0$ OR
  $P$ is the minimal prime of $(H^j_P(M_n))$. Thus we get that $L_n$ is supported  ONLY at the maximal ideal of $B$ for all $n \in \ZZ$. By \ref{support-m} we get that $L_n = E_B(K)^{\alpha_n}$ where 
 $E_B(K)$ is the injective hull of $K$ as a $B$-module (and $\alpha$ some ordinal possibly infinite). We note that
 \begin{enumerate}
 \item
 $\alpha_n = \mu_j(P, M_n)$.
 \item
 $\alpha_c < \infty$.
 \end{enumerate}
    By \ref{key-remark} we get that
$V = H_d(\bY, L)$ is a graded holonomic, generalized Eulerian $A_m(K)$-module with $V_n = H_d(\bY, L)_n = K^{\alpha_n}$ for all $n \in \ZZ$.

Now $\dim_K V_c = \alpha_c < \infty$. By Proposition \ref{bass-dim0} we get $\dim_K V_n < \infty $ for all $n \in \ZZ$. It follows that $\mu_j(P, M_n) = \dim_K V_n < \infty$ for all $n \in \ZZ$.  Finally we note that the assertions (a), (b), (c), (d) and (e) follow from Theorems \ref{Weyl} \ref{rig-tame-dim-0} and  \ref{rig-tame-suprise-dim-0}.
 \end{proof}
 
 \section{Proof of Theorem \ref{bass-m-one}}
In this section we give a proof of Theorem \ref{bass-m-one}. The following result is crucial.
\begin{lemma}\label{lemma-bass-m-one}
Let $M = \bigoplus_{n \in \ZZ } M_n$ be a graded holonomic, generalized Eulerian $A_1(K)$-module. Then $\dim_K M_n < \infty$ for all $n \in \ZZ$.
\end{lemma}
 \begin{proof}
By Proposition \ref{bass-dim0} it suffices to show $\dim_K M_0 < \infty$. Let 
$D = A_1(K) = K<X, \partial>/(\partial X - X \partial -1)$. An element $d$ of $D$ is uniquely expressed as $\sum \alpha_{ij} X^i\partial^j$ with $\alpha_{ij} \in K$ and $\alpha_{ij} = 0$ for all but finitely many $i,j$. Thus  if $d \in D_0$ the degree zero component of $D$, then $d$ is a finite  linear sum of 
$\{ X^i\partial^i ; i \geq 0 \}$. 

Let $\E = X\partial$ be the Eulerian operator. Then it is well-known that
$X^i\partial^i$ can be expressed as a polynomial in $\E$. In fact by 
\cite[Lemma 1.3.1]{SST}
for $j \geq 2$ we have
\[
X^j\partial^j = \E(\E -1)\ldots (\E - j + 1).
\]
It follows that $D_0 = K[\E]$.

Let $V$ be a $D_0$-submodule of $M_0$. It is easy to prove that $(D V) \cap M_0  = V$. \\
\textit{Claim:} $M_0$ is a finitely generated $D_0$-module.\\
Let $$V_1 \subseteq V_2 \subseteq  \cdots \subseteq V_l \subseteq V_{l+1} \subseteq  \cdots, $$
be an ascending chain of $D_0$-submodules of $M_0$.
Then we have an ascending chain of $D$-submodules of $M$;
$$DV_1 \subseteq DV_2 \subseteq  \cdots \subseteq DV_l \subseteq DV_{l+1} \subseteq  \cdots. $$
As $M$ is a Noetherian $D$-module there exists $r$ such that $DV_j = DV_r$ for all $j \geq r$. It follows that
\[
V_j = (DV_j)\cap M_0 = (DV_r)\cap M_0 = V_r; \quad \text{for all} \ j \geq r.
\]
Thus $M_0$ is a Noetherian $D_0$-module. In particular it is finitely generated as a $D_0$-module.

Let $M_0$ be generated as a $D_0$-module by $\{ u_1,\ldots, u_s \}$. Now $M$ is a
generalized  Eulerian $D$-module. In particular there exists $r_i$ such that
\[
(\E)^{r_i}u_i = (\E - |u_i|)^{r_i}u_i  = 0 \quad \text{for all} \ i.
\]
Let $r = \max \{ r_i \}$. Let $v \in M_0$. Then
\[
u = \alpha_1 u_1 + \cdots + \alpha_s u_s \quad \text{where} \ \alpha_i \in K[\E].
\]
It follows that $\E^r u = 0$. Thus $M_0$ is a finitely generated $k[\E]/(\E^r)$-module. In particular $\dim_K M_0 < \infty$.
\end{proof}
We now give
\begin{proof}[Proof of Theorem \ref{bass-m-one}]
We do the same construction as in the proof of Theorem \ref{bass-basic-proof}.
 By Lemma \ref{lemma-bass-m-one}   we get $ \alpha_n = \dim_K V_n < \infty $ for all $n \in \ZZ$. It follows that $\mu_j(P, M_n) < \infty$ for all $n \in \ZZ$.
\end{proof}
\section{Growth of Bass numbers}
In this section we prove Theorems \ref{bass-growth} and Theorem \ref{bass-growth-max}. To prove Theorem \ref{bass-growth} we need to prove the following result:
\begin{theorem}
\label{GB-hol} Let $M = \bigoplus_{n \in \ZZ}M_n$ be a graded holonomic generalized Eulerian $A_m(K)$-module with $\dim_K M_n < \infty$ for all $n \in \ZZ$. Then there exists polynomials $P_M(z), Q_M(z)\in \mathbb{Q}[z]$  of degree $\leq m -1$ such that 
\begin{align*}
P_M(n) &= \dim_K M_n  \ \text{for all} \  n \ll 0, and \\
Q_M(n) &= \dim_K M_n  \ \text{for all} \  n \gg 0.
\end{align*}
\end{theorem}
\begin{remark}
By our convention degree of the zero polynomial is $-\infty$.
\end{remark}
\begin{proof}[Proof of Theorem \ref{GB-hol}]
We prove the result by induction on $m$. We first assume $m = 1$.

We have an exact sequence
 \begin{equation}
  \label{bass-induct-1}
  0 \rt H_1(\partial_1, M)_j \rt M_{j+1} \rt M_j \rt H_0(\partial_1, M)_j \rt 0.
 \end{equation}
By \ref{sushil} we get that  for $l = 0,1$; $H_l(\partial_1, M)$ is concentrated in degree $-1$.
So by \ref{bass-induct-1} it follows that 
\[
 M_j \cong M_{-1} \quad \text{for all} \ j \leq -1.
\]
Set $P_M(z) = \dim_K M_{-1}$.

 We  also have an exact sequence
 \begin{equation}
  \label{bass-induct-3}
  0 \rt H_1(X_1, M)_j \rt M_{j-1} \rt M_j \rt H_0(X_1, M)_j \rt 0.
 \end{equation}
By \ref{sushil} we get that for $l = 0, 1$; $H_l(X_1, M)$ is concentrated in degree $0$.
So by \ref{bass-induct-3} it follows that 
\[
 M_j \cong M_{0} \quad \text{for all} \ j \geq 0.
\]
Set $Q_M(z) = \dim_K M_{0}$.
Thus we have the result for $m = 1$.

Assume $m \geq 2$ and the result is known for $m -1$.
We have an exact sequence
 \begin{equation}
  \label{bass-induct-2}
  0 \rt H_1(\partial_m, M)_j \rt M_{j+1} \rt M_j \rt H_0(\partial_m, M)_j \rt 0.
 \end{equation}

By \ref{sushil} $H_l(\partial_m; M)(-1)$ is generalized Eulerian $A_{m-1}(K)$-module for $l = 0, 1$. Furthermore  $ \dim_K H_l(\partial_m; M)_j < \infty$  for all $j \in \ZZ$ and $l = 0, 1$.  So by induction hypothesis it follows that there exists polynomials $f(z), g(z)$ of degree $\leq m-2$ such that  for all $n \ll 0$, 
\[
f(n) = \dim_K H_1(\partial_m; M)_n  \   \text{and} \ g(n) = \dim_K H_0(\partial_m; M)_n. 
\]
By \ref{bass-induct-2} we have 
\[
\dim_K M_{n+1} - \dim_K M_n = f(n) - g(n) \quad \text{for all } \ n \ll 0.
\]
It follows that the function $n \rt \dim_K M_n$ (with $n < 0$) is of polynomial type of degree $\leq m -1$.

We also have an exact sequence
 \begin{equation}
  \label{bass-induct-4}
  0 \rt H_1(X_m, M)_j \rt M_{j-1} \rt M_j \rt H_0(X_m, M)_j \rt 0.
 \end{equation}

By \ref{sushil} $H_l(X_m; M)$ is generalized Eulerian $A_{m-1}(K)$-module for $l = 0,1$.  Furthermore  $ \dim_K H_l(X_m; M)_j < \infty$  for all $j \in \ZZ$ and $l = 0, 1$. 
So by induction hypothesis it follows that there exists polynomials $h(z), t(z)$ of degree $\leq m-2$ such that  for all $n \gg 0$, 
\[
h(n) = \dim_K H_1(X_m; M)_n  \   \text{and} \ t(n) = \dim_K H_0(X_m; M)_n. 
\]
By \ref{bass-induct-4} we have 
\[
\dim_K M_{n} - \dim_K M_{n-1} = t(n) - h(n) \quad \text{for all } \ n \gg 0.
\]
It follows that the function $n \rt \dim_K M_n$ (with $n > 0$) is of polynomial type of degree $\leq m -1$.
\end{proof}

We now give
\begin{proof}[Proof of Theorem \ref{bass-growth}]
By Lemma \ref{lyu-lemma-hypoth} and Proposition \ref{lyu-lemma} we get 
that \\ $\mu_j(P, M_n) = \mu_0(P, H^j_P(M_n))$ for all $n \in \ZZ$.
We note that $(H^j_{PR}\circ \FF)(R)_n = H^j_P(M_n)$ for all $n \in \ZZ$. We note that $H^j_{PR}\circ\FF$ is a graded Lyubeznik functor on $\ ^*Mod(R)$.

Let $B = \widehat{A_P}$. Also 
 set $S = B[X_1,\ldots, X_m]$.
 We note that by \ref{std-op} the functor $G(-) = B \otimes_A H^j_{PR}\circ\FF(-)$ is a graded Lyubeznik functor on $ \ ^* Mod(S)$.
 
 By Cohen-structure theorem $B = K[[Y_1,\ldots, Y_g]]$ where $K = \kappa(P)$ the residue field of $A_P$  and $g = \height_A P$.
 Let $\Lambda $ be the ring of $K$-linear differential operators on $B$ and set $\D = A_m(\Lambda)$ the $m^{th}$-Weyl algebra
 over $\Lambda$. Then by Theorem \ref{gen-eul-Lyu} and Theorem \ref{mom} we get that $L = G(R)$ is a graded holonomic, generalized Eulerian
 $\D$-module. Notice  by \ref{min-loc} we get that 
 $L_n = (H^j_P(M_n))_P $ for all $n \in \ZZ$.
 
 Note that either $(H^j_P(M_n))_P  = 0$ OR
  $P$ is the minimal prime of $(H^j_P(M_n))$. Thus we get that $L_n$ is supported  ONLY at the maximal ideal of $B$ for all $n \in \ZZ$. By \ref{support-m} we get that $L_n = E_B(K)^{\alpha_n}$ where 
 $E_B(K)$ is the injective hull of $K$ as a $B$-module (and $\alpha_n$ some ordinal possibly infinite).  By \ref{key-remark} we get that
$V = H_d(\bY, L)$ is a graded holonomic generalized Eulerian $A_m(K)$-module with $V_n = H_d(\bY, L)_n = K^{\alpha_n}$ for all $n \in \ZZ$.
We note that
  that $ \alpha_n = \mu_j(P, M_n)  < \infty$ for all $n \in \ZZ$.
  
  The result now follows from Theorem \ref{GB-hol}.
\end{proof}
\s \label{hilb} Before proving Theorem \ref{bass-growth-max} we need the following preliminaries. 
\begin{enumerate}
\item
Let $T = K[X_1,\ldots, X_m] = \bigoplus_{n \geq 0} T_n$. Then it is well-known that
\[
\dim_K T_n = \binom{n + m -1}{m-1} \quad \text{for all} \ n \geq 0.
\]
\item
Let $\n = (X_1,\ldots, X_m)$. Then $H^m_\n(T) = E(m)$ where $E$ is the $*$-injective hull of $\n$ as a $R$-module. Using for instance the Grothendieck-Serre formula \cite[4.4.3]{BH} we get
\[
\dim_K H^m_\n(T)_n = (-1)^{m-1} \binom{n + m -1}{m-1} \quad \text{for all} \ n \leq -m.
\]
\end{enumerate}

The following result is an essential ingredient in proving Theorem \ref{bass-growth-max}.
\begin{proposition} \label{msudan}
Let $M = \bigoplus_{n \in \ZZ}M_n$ be a graded holonomic generalized Eulerian $A_m(K)$-module. Suppose $M_c = 0$ for some $c$ but $M \neq 0$.
Then one of the following holds:
\begin{enumerate}[\rm(1)]
\item
$\dim_K M_0 $ is finite and non-zero, 
$$\dim_K M_n = \dim_K M_0 \cdot \binom{n + m -1}{m-1} \quad \text{for all $n \geq 0$}$$
 and $M_n = 0$ for all $n < 0$.
\item
$\dim_K M_{-m} $ is finite and non-zero, 
$$\dim_K M_n = \dim_K M_{-m} \cdot (-1)^{m-1}\binom{n + m -1}{m-1} \quad \text{ for all $n \leq -m$} $$ and $M_n = 0$ for all $n > -m$.
\end{enumerate}  
\end{proposition}
\begin{proof}
We first note by  Proposition \ref{bass-dim0} we get that $\dim_K M_n < \infty$ for all $n \in \ZZ$. By Theorem \ref{rig-tame-dim-0} and Theorem \ref{rig-tame-suprise-dim-0} we also get that we have one of the following two cases:
\begin{enumerate}[\rm (i)]
\item
$M_n \neq 0$ for all $n \geq 0$ and $M_n = 0$ for all $n < 0$.
\item
$M_n \neq 0$ for all $n \leq -m$ and $M_n = 0$ for all $n > -m$.
\end{enumerate}
Thus all we have to show is the result regarding dimensions of graded components of $M$. For convenience set $D = A_m(K)$. Also we consider $T = K[X_1,\ldots, X_m]$ with its unique maximal homogeneous  ideal $\n = (X_1,\ldots, X_m)$. Furthermore 
let $E = D/D\n$ be the injective hull of $K = T/\n$ as a $T$-module. We note that
$E(m)$ is an Eulerian $D$-module.

(1) Let $p = \dim_K M_0$ and let $u_1,\ldots, u_p$ be a $K$-basis of $M_0$. We consider the map $\psi \colon D^p \rt M$ which maps $e_i$ to $u_i$. Notice that
$\psi ((D\partial)^p) = 0$. So $\psi$ factors through a map $\ov{\psi} \colon T^p \rt M$. Thus we have an exact sequence of generalized Eulerian holonomic $D$-modules
\[
0 \rt K \rt T^p \xrightarrow{\ov{\psi}} M \rt C \rt 0 \quad \text{where} \ K = \ker \ov{\psi}, \ \text{and} \ C = \coker \ov{\psi}.
\]
We note that by construction, $K_j = C_j = 0$ for $j < 1$. By Theorem \ref{rig-tame-dim-0} it follows that $K = C = 0$. Thus $M = T^p$. The result now follows from \ref{hilb}(1).

(2) Let $q = \dim_K M_{-m}$ and let $v_1,\ldots, v_q$ be a $K$-basis of $M_{-m}$. We consider the map $\phi \colon D^q \rt M$ which maps $e_i$ to $v_i$. Notice that
$\phi ((D\n)^q) = 0$. So $\phi$ factors through a map $\ov{\phi} \colon E(m)^q \rt M$. Thus we have an exact sequence of generalized Eulerian holonomic $D$-modules
\[
0 \rt K^\prime \rt E(m)^q \xrightarrow{\ov{\phi}} M \rt C^\prime \rt 0 \quad \text{where} \ K^\prime = \ker \ov{\phi}, \ \text{and} \ C^\prime = \coker \ov{\phi}.
\]
We note that by construction, $K^\prime_j = C^\prime_j = 0$ for $j > -m - 1$. By Theorem \ref{rig-tame-dim-0} it follows that $K^\prime = C^\prime = 0$. Thus $M = E(m)^q$. The result now follows from \ref{hilb}(2).
\end{proof}

We now give
\begin{proof}[Proof of Theorem \ref{bass-growth-max}]
The proof proceeds on the same lines as proof of Theorem \ref{bass-growth}. We only note that in this case we have
$V = H_d(\bY, L) \neq 0$  with $V_c = 0$ for for some $c$. 
The result now follows from Proposition \ref{msudan}.
\end{proof}

\section{Associate Primes}
In this section we prove Theorem \ref{ass}.
 To prove this theorem we need to generalize  an exercise problem from 
Matsumura's  classic text \cite[Exercise 6.7]{Mat}.
\begin{proposition}\label{M-ex}
Let $f \colon A \rt B$ be a homomorphism of Noetherian rings. Let $M$ be an $B$-module. Then 
\[
\Ass_A M  = \{ P\cap A \mid P \in \Ass_B M \}.
\]
In particular if $\Ass_B M$ is a finite set then so is $\Ass_A M$.
\end{proposition}
\begin{remark}
Matsumura's exercise is to prove the above result for finitely generated $B$-modules. We assume that the reader has done this exercise.  We also believe that Proposition \ref{M-ex} is known to the experts.  We give a proof for reader's convenience.

\end{remark}
\begin{proof}[Proof of Proposition \ref{M-ex}]
It can be easily proved that
\[
\{ P\cap A \mid P \in \Ass_B M \} \subseteq \Ass_A M.
\]
Let $\q \in \Ass_A M$. So $\q = (0 \colon t)$ for some non-zero $t \in M$. Let $N = Bt$. Note that $N$ is a finitely generated $B$-module and $\q \in \Ass_A N$. As the assertion of the Proposition is true for finitely generated $B$-modules, there exists
$P \in \Ass_B N$ with $P\cap A = \q$. As $N \subseteq M$ we also have $P \in \Ass_B M$. 
\end{proof}

\s We note that $P \in \Ass_A V$ if and only if $\mu_0(P, V) > 0$.
 We now state and prove a result which implies Theorem \ref{ass}.
 \begin{theorem}\label{ass-gen}
 (with hypotheses as in \ref{std}). 
  Further assume that either $A$ is local or a smooth affine algebra over a field $K$ of characteristic zero. Let  $\FF$ be a graded Lyubeznik functor of $ \ ^* Mod(R)$. Set $M = \FF(R) = \bigoplus_{n \in \ZZ}M_n$. Then
\begin{enumerate}[\rm (1)]
\item
$\bigcup_{n \in \ZZ} \Ass_A M_n   $ is a finite set.
\item
$\Ass_A M_n = \Ass_A M_m$ for all $n \leq -m$.
\item
$\Ass_A M_n = \Ass_A M_0$ for all $n \geq 0$.
\end{enumerate}
\end{theorem}
 \begin{proof}
 We first show that under our assumptions $\Ass_R \FF(R)$ is finite.
 
 If $A$ is a smooth affine algebra over $K$ then so is $R = A[X_1,\ldots, X_m]$. In this case if $G$ is \emph{any} Lyubeznik functor on $Mod(R)$ (not necessarily graded) then $\Ass_R G(R)$ is finite, see \cite[3.7]{Lyu-1}.
 
 Now assume $A$ is local with maximal ideal $\n$. Let $\M = (\n, X_1,\ldots, X_m)$ be the maximal homogeneous ideal of $R$. As $\FF(R)$ is a graded $R$-module all its associate primes are homogeneous (see \cite[1.5.6]{BH}) and  so are contained in $\M$. Thus we have an isomorphism $\Ass_R(\FF(R)) \rt \Ass_{R_\M}(\FF(R)_\M)$. But
 $\FF(R)_{R_\M} = G(R_\M)$ for a Lyubeznik functor $G$ on $Mod(\R_\M)$, see \ref{flat-L}. However $R_\M$ is regular local. Thus the result follows from \cite[3.3]{Lyu-1}.
 
 Thus we have proved that under our assumptions $\Ass_R \FF(R)$ is finite.
 
 (1) This follows from Proposition \ref{M-ex}.
 
 For (2), (3) let 
 \[
 \bigcup_{n \in \ZZ} \Ass_A M_n  = \{ P_1, \ldots, P_l \}.
 \]
 
 (2) Let $P = P_i$ for some $i$. 
 Let $r \leq -m$. Then by Theorem  \ref{bass-basic} it follows that $\mu(P, M_r) > 0$ if and only if $\mu(P, M_{-m}) > 0$.   The result follows.
 
 (3)Let $P = P_i$ for some $i$. 
 Let $s \geq 0$. Then by Theorem  \ref{bass-basic} it follows that $\mu(P, M_s) > 0$ if and only if $\mu(P, M_{0}) > 0$.   The result follows.
 \end{proof}
 
 We now state a more general result which  essentially has the same proof as of Theorem \ref{ass-gen}. First we make the following notation: Let $\m$ be a maximal ideal of $A$. If $N$ is an $A$-module, set
 \[
 \Ass_A^\m(N) = \{ P \in \Ass_A N \mid P \subseteq \m \}.
 \]
 \begin{theorem} \label{ass-gen2}
 (with hypotheses as in \ref{std}). 
Let $\m$ be a maximal ideal in $A$  
  Let  $\FF$ be a graded Lyubeznik functor of $ \ ^* Mod(R)$. Set $M = \FF(R) = \bigoplus_{n \in \ZZ}M_n$. Then
\begin{enumerate}[\rm (1)]
\item
$\bigcup_{n \in \ZZ} \Ass_A^\m M_n   $ is a finite set.
\item
$\Ass_A^\m M_n = \Ass_A^\m M_m$ for all $n \leq -m$.
\item
$\Ass_A^\m M_n = \Ass_A^\m M_0$ for all $n \geq 0$. \qed
\end{enumerate}
 \end{theorem}
\section{Dimension of Support and injective dimension}
\s \label{set-final} \emph{Setup} For the convenience of the reader we recall our basic assumptions. $A$ will denote a regular ring containing a field of characteristic zero. Let $R = A[X_1,\ldots, X_m]$ be standard graded with $\deg A = 0$ and $\deg X_i = 1$ for all $i$. We also assume $m \geq 1$. Let $\FF(-)$ be a graded Lyubeznik functor on $ \ ^* Mod(R)$.  Set $M =  \FF(R) = \bigoplus_{n \in \ZZ}M_n$.

Let $E$ be an $A$-module. Let
$\injdim_A E$ denotes the injective dimension of $E$. Also 
$\Supp_A E = \{ P \mid  E_P \neq 0 \ \text{and $P$ is a prime in $A$}\}$ is the support of an $A$-module $E$. 
By $\dim_A E $ we mean the dimension of $\Supp_A E$ as a subspace of $\Spec(A)$.

We first show
\begin{lemma}\label{injdim-dim}
(with hypotheses as in \ref{set-final}). Let $c \in \ZZ$. Then
\[
\injdim M_c \leq \dim M_c.
\]
\end{lemma}
\begin{proof}
Let $P$ be a prime ideal in $A$. Then by Proposition \ref{lyu-lemma-hypoth} and Lemma \ref{lyu-lemma} we get
\[
\mu_j(P, M_c) = \mu(P, H^j_P(M_c)).
\]
By Grothendieck vanishing theorem $H^j_P(M_c) = 0$ for all $j > \dim M_c$, see \cite[6.1.2]{BS}.  So $\mu_j(P, M_c) = 0$ for all $j > \dim M_c$. The result follows.
\end{proof}
We now state and prove a result which implies Theorem \ref{injdim-and-dim}.
\begin{theorem}\label{injdim-and-dim-gen}(with hypotheses as in \ref{set-final}). Let $M = \FF(R) = \bigoplus_{n \in \ZZ}M_n$. Then we have
\begin{enumerate}[\rm (1)]
\item
$\injdim M_c \leq \dim M_c$ for all $c \in \ZZ$.
\item
$\injdim M_n = \injdim M_{-m}$ for all $n \leq -m$.
\item
$\dim M_n = \dim M_{-m}$ for all $n \leq -m$.
\item
$\injdim M_n = \injdim M_{0}$ for all $n \geq 0$.
\item
$\dim M_n = \dim M_{0}$ for all $n \geq 0$.
 \item
 If $m \geq 2$ and $-m < r,s < 0$ then
 \begin{enumerate}[\rm (a)]
 \item
 $\injdim M_r = \injdim M_{s}$  and $\dim M_r = \dim M_{s}$.
 \item
 $\injdim M_r \leq \min \{ \injdim M_{-m}, \injdim M_{0} \}$.
 \item
 $\dim M_r \leq \min \{ \dim M_{-m}, \dim M_{0} \}$.
 \end{enumerate}
\end{enumerate}
\end{theorem}
\begin{proof}
(1) This follows from Lemma \ref{injdim-dim}.

For (2), (3) let $P$ be a prime ideal in $A$.  Let $c \leq -m$.

(2) Fix $j \geq 0$. By
Theorem \ref{bass-basic-proof} we get that $\mu_j(P, M_c)> 0$ if and only if 
$\mu_j(P, M_{-m})> 0$. The result follows.

(3) We note that $\FF(-)_P$ is a graded Lyubeznik functor on \\  $\ ^* Mod(A_P[X_1,\ldots, X_m])$. By Theorem \ref{tame-rigid-general} it follows that $(M_{-m})_P \neq 0$ if and only if $(M_c)_P \neq 0$. The result follows.

(4), (5), 6(a) follow with similar arguments as in (2), (3).

For 6(b),(c)
let $P$ be a prime ideal in $A$. 

6(b) By
Theorem \ref{bass-basic-proof} we get that if $\mu_j(P, M_r)> 0$ then 
$\mu_j(P, M_{-m})> 0$ and  $\mu_j(P, M_{0})> 0$. The result follows.

6(c) We note that $\FF(-)_P$ is a graded Lyubeznik functor on \\  $\ ^* Mod(A_P[X_1,\ldots, X_m])$. By Theorem \ref{tame-rigid-general} it follows that $(M_{r})_P \neq 0$  implies \\ $(M_{-m})_P \neq 0$ and $(M_{0})_P \neq 0$ . The result follows.
\end{proof}


\begin{thebibliography}{10}

\bibitem{Bjork} J.-E. Bj$\ddot{o}$rk, \emph{Rings of differential operators}, North-Holland Math. Library \textbf{21}, North Holland, Amsterdam, 1979. 

\bibitem{BS}  
M.~P.~Brodmann and R.~Y.~Sharp, 
\emph{Local cohomology: an algebraic introduction with geometric applications}, Cambridge Studies in Advanced Mathematics 60, Cambridge University Press,~Cambridge, 1998.

\bibitem{BH}
W.~Bruns and J.~Herzog, \emph{{Cohen-Macaulay rings}}, 
 Edition, vol.~39, Cambridge
  studies in advanced mathematics, Cambridge University Press,~Cambridge, 1997.

\bibitem{CH}  
S.~D.~Cutkosky and J.~ Herzog, 
\emph{Failure of tameness for local cohomology},
J. Pure Appl. Algebra 211 (2007), no. 2, 428–-432. 

\bibitem{HuSh}
C.~Huneke and R.~Sharp,
\emph{ Bass Numbers of Local Cohomology Modules}, 
AMS Transactions 339 (1993), 765–-779. 

\bibitem{Lam}
T.~Y.~Lam,
\emph{A first course in noncommutative rings},
Graduate Texts in Mathematics, 131. Springer-Verlag, New York, 1991


\bibitem{Lyu-1}
G.~Lyubeznik, 
\emph{Finiteness Properties of Local Cohomology Modules (an Application of D-modules to Commutative Algebra)},
 Inv. Math. 113 (1993), 41–-55.
 
\bibitem{Lyu-2}
\bysame,
\emph{F-modules: applications to local cohomology and D-modules in characteristic p>0}, 
J. Reine Angew. Math. 491 (1997), 65–-130.

 
 \bibitem{MaZhang} L.~Ma and W.~Zhang, \emph{Eulerian graded D-modules}, Math. Res. Lett. \textbf{21} (2014), no. 1, 149–167.
 
 \bibitem{Mat}
H.~Matsumura, \emph{Commutative ring theory}, second ed., Cambridge
  Studies in Advanced Mathematics, vol.~8, Cambridge University Press,
  Cambridge, 1989. 

 
 \bibitem{P1}  T. J. Puthenpurakal, \emph{De Rham cohomology of local cohomology modules}, Preprint.arXiv:
1302.0116v1 to appear in \textit{Advances in Algebra and its Applications}: Aligarh, India, December 2014; Editors: S. Tariq Rizvi, Asma Ali and Vincenzo De Filippis .

 
 \bibitem{P2} T.~J.~Puthenpurakal, \emph{De Rham cohomology of local cohomology modules-The graded case}, Nagoya Math. J. \textbf{217} (2015), 1-21.
 
 \bibitem{PR1} 
 T.~J.~Puthenpurakal and R.~B.~T.~Reddy, 
\emph{de Rahm Cohomology of Local Cohomology modules II},
Preprint: arXiv:1308.0165 
 
 \bibitem{PR2}
 T.~J.~Puthenpurakal,  R.~B.~T.~Reddy, 
  \emph{de Rham cohomology of $H^1_{(f)}(R)$ where $V(f)$ is a smooth hypersurface in $\mathbb{P}^n$},
Preprint: arXiv:1310.4654 

  \bibitem{PS} T.~J.~Puthenpurakal and J.~Singh,
  \emph{On derived functors  of Graded local cohomology modules},
  Preprint.arXiv:1612.02968
  
  \bibitem{SST}
  M.~Saito, B.~Sturmfels and N.~Takayama, 
\emph{Gr\"{o}bner deformations of hypergeometric differential equations},
Algorithms and Computation in Mathematics, 6. Springer-Verlag, Berlin, 2000. 
\end{thebibliography}
\end{document}